\documentclass[12 pt, a4paper, oneside]{amsart}
\usepackage{indentfirst,amsmath,amssymb,amsthm,fullpage,amsfonts,enumerate,mathrsfs,tabularx,graphicx, xcolor, tikz-cd,hyperref, setspace, upgreek, mathabx, titletoc, mathrsfs, mathtools, float, subfig, amsthm, dsfont,marginnote,pdfcomment,multicol,xfrac}
\usetikzlibrary{matrix,arrows,backgrounds,positioning}
\usepackage{etoolbox}
\makeatletter
\patchcmd{\@thm}{\thm@headfont{\scshape}}{\thm@headfont{\scshape\bfseries}}{}{}
\patchcmd{\@thm}{\thm@notefont{\fontseries\mddefault\upshape}}{}{}{}
\makeatother

\makeatletter
\@addtoreset{figure}{section}
\makeatother

\newtheorem{theorem}{Theorem}[section]
\newtheorem{proposition}[theorem]{Proposition}
\newtheorem{lemma}[theorem]{Lemma}
\newtheorem{corollary}[theorem]{Corollary}
\newtheorem{remark}[theorem]{Remark}
\newtheorem{definition}[theorem]{Definition}

\newtheorem{situation}[theorem]{Situation}
\newtheorem{example}[theorem]{Example}
\newcommand{\integers}{\mathbb{Z}}

\newcommand{\complexns}{\mathbb{C}}

\newcommand{\comment}[1]{}

\newcommand*{\pic}{\text{Pic}}

\newcommand{\etale}{étale}
\renewcommand{\hom}{\text{Hom}}

\newcommand{\projective}{\mathds{P}}

\newcommand{\etalecohomology}[3]{\text{H}^{#1}_{\acute{e}t}(#2,#3)}

\newcommand{\aut}{\text{Aut}}

\newcommand{\cohomology}[3]{\text{H}^{#1}(#2,#3)}

\newcommand{\sheaves}{\mathit{Sh}}
\newcommand{\absheaves}{\textit{Ab}}
\newcommand{\morphisms}{\textit{Mor}}

\newcommand{\presheaves}{\textit{PSh}}

\title{Torsors and Gerbes}
\author{Ashwin Deopurkar}
\address{Ashwin Deopurkar}
\email{ard2144@columbia.edu}
\date{\today}

\begin{document}    
\maketitle
\section{Introduction} \label{section_torsors_gerbes}
\newcommand{\objects}{\text{Ob}}
\newcommand{\invariants}{A^{G}}
\newcommand{\autgroup}{\integers/ 2 \integers}
\newcommand{\xdown}{X}
\newcommand{\xup}{\widetilde{X}}
\newcommand{\isom}{\text{Isom}}
\newcommand{\gerbe}{\mathscr{Y}}
\newcommand{\band}{\text{Band}}
\newcommand{\thesite}{X_{\tau}}
\newcommand{\gerbeup}{\widetilde{A}\emph{-torsors}}
\newcommand{\gerbedown}{\mathscr{Y}}
\newcommand{\trivialgerbe}{\mathscr{Y}^{triv}}
\newcommand{\trivialaggerbe}{\overline{\mathscr{Y}}^{triv}}
\newcommand{\sheafonc}{\widetilde{A}_{\mathscr{C}}} 
\newcommand{\gerbeonc}{\gerbe_{\mathscr{C}}}
\newcommand{\gerberho}{\mathscr{Y}_{\rho}}
\newcommand{\sheafaonc}{A_{\mathscr{C}}}
\newcommand{\gloabatilda}{\widetilde{A}(\xup)}
\newcommand{\zgerbe}{\mathscr{Z}}
\newcommand{\fibprocat}{\mathscr{C}}
\newcommand{\gerbetilda}{\widetilde{\mathscr{Y}}}
\newcommand{\etalesite}{X_{\text{ét}}}
\newcommand{\torsorcat}[1]{\underline{A\mbox{-torsors}}}
\newcommand{\basecat}{\mathscr{C}}
\newcommand{\gmsheaf}{\mathbf{G}_{m}}

\begin{sloppypar}
  Let $\thesite$ be a site with a final object $\xdown$. Let $A$ be an abelian sheaf on $\thesite$ on which we have an action of a finite group $G$. Let $\invariants$ denote the sheaf of invariants. Then given an $A\mbox{-torsor}$ (respectively a gerbe banded by $A$), we would like to know under what conditions it is induced from an $\invariants\mbox{-torsor}$ (respectively a gerbe banded by $\invariants$) (Definition \ref{def-induced-torsor-and-gerbe}). It is easy to see that it is necessary for the torsor or the gerbe to admit a lift of the group action (Definition \ref{definition-G-action-on-torsor}, Definition \ref{def-group-action-on-gerbe}). Given such a $G\mbox{-action}$, we give explicit construction of an obstruction class which is a cohomology class in ${\cohomology{2}{G}{A(\xdown)}}$ for torsors and in ${\cohomology{3}{G}{A(\xdown)}}$ for gerbes. We prove that if a torsor or a gerbe is induced from the invariants then the obstruction class is zero (Proposition \ref{prop-induced-torsor-and-gerbes}). Moreover, the vanishing of the obstruction class is sufficient for the reverse conclusion for torsors when the first group cohomology classes locally vanish (Definition \ref{def-group-cohomlogies-vanish-locally}). Similarly for gerbes, the vanishing of the obstruction class in ${\cohomology{3}{G}{A(\xdown)}}$ is sufficient if we assume that the first two group cohomology classes locally vanish (Theorem \ref{theorem-induced-from-invariants-when-cohomology-class-is-zero}). The first two sheaf cohomology groups $\cohomology{i}{\xdown}{A}$ can be realized concretely as the group of $A\mbox{-torsors}$ and the group of abelian gerbes banded by $A$. Using this realization and the notion of obstruction classes, we describe morphisms of cohomologies in low degrees which arise from a variant of the Hoschild-Serre spectral sequence.

  As an application we obtain the following result (Proposition \ref{prop_galois_cover_line_bundle_exists}). Let $\xdown$ be a scheme such that $\etalecohomology{2}{\xdown}{\gmsheaf}$ is zero. Suppose ${\xup \xrightarrow{\pi} \xdown}$ is a finite Galois \etale{} cover with Galois group $G$. Then for any homomorphism ${\Phi: G \to \pic(\xdown)}$, the group cohomology class of ${\{\pi^*\Phi(g)\}_{g \in G}}$ in ${\cohomology{1}{G}{\pic(\xup)}}$ is zero. In the case of a double cover of a curve over an algebraically closed field, this was observed by Mumford in the context of theta characteristics (See Lemma 1, \cite{mumford_theta_characteristics}). Since our proof uses purely cohomological methods, it also applies when $\xdown$ is a stacky curve (for example a twisted curve) as long as ${\etalecohomology{2}{\xdown}{G_m}}$ is zero. Although we do not delve into it in this article, these considerations are related to cup products in sheaf cohomology (See Section 2, \cite{antieau2020brauer}). We use this to study Weil pairing on twisted curves in our upcoming paper.
\end{sloppypar}

\section{Preliminaries}\label{section-Preliminaries}
We use notations and conventions as in Stacks Project (\cite[\href{https://stacks.math.columbia.edu/tag/0013}{Tag 0013}]{stacks-project}) for categories and sites (\cite[\href{https://stacks.math.columbia.edu/tag/00VG}{Tag 00VG}]{stacks-project}). Let $\thesite$ be a site with a final object $\xdown$. Let $A$ be an abelian sheaf on $\thesite$. By $\cohomology{i}{\xdown}{-}$ we mean the sheaf cohomology on the site $\thesite$. 
\begin{definition}\label{definition-a-torsor}
  \begin{sloppypar}
    An $A\mbox{-torsor}$ $y$ on $\xdown$ is a sheaf of sets $y$ on $\thesite$ together with a group action ${A \times y \to y}$ such that
  \end{sloppypar}
\begin{itemize}
\item For every object $U$ in $\thesite$, the map $A(U) \times y(U) \to y(U)$ is simply transitive whenever $y(U)$ is non-empty.
\item For every object $U$ in $\thesite$, there exists a cover $\{U_i \to U\}_{i \in I}$ such that $y(U_i)$ is non-empty for every $i \in I$. 
\end{itemize}
We say $y$ is a pseudo $A\mbox{-torsor}$, if it satisfies the above conditions except possibly the last. The sheaf $A$ itself is an $A\mbox{-torsor}$ with the action given by the group law. We refer to $A$ as the trivial torsor.     
\end{definition}
\begin{definition}
  \begin{sloppypar}
    Let $A_1$ and $A_2$ be abelian sheaves on $\xdown$. Let $y_1$, $y_2$ be an $A_1\mbox{-torsor}$ and an $A_2\mbox{-torsor}$, respectively. Suppose we have a map ${f: A_1 \to A_2}$ of abelian sheaves. Then an $f\mbox{-equivariant}$ morphism ${\lambda : y_1 \to y_2}$ is a morphism of sheaves such that the following diagram commutes.
  \end{sloppypar}
  \[
    \begin{tikzcd}
      A_1 \times y_1 \arrow{d}{f \times \lambda} \arrow{r} & y_1 \arrow{d}{\lambda} \\
      A_2 \times y_2 \arrow{r} & y_2 
    \end{tikzcd}
  \]
An $A\mbox{-equivariant}$ morphism or a morphism of $A\mbox{-torsors}$ is a morphism of sheaves which is equivariant with respect to the identity map from $A$ to itself. 
\end{definition}

\begin{sloppypar}
  The category $\torsorcat{A}$ is defined as follows (See \cite[\href{https://stacks.math.columbia.edu/tag/036Z}{Tag 036Z}]{stacks-project}). The objects of $\torsorcat{A}$ are pairs $(U,y)$ where $U$ is an object of $\thesite$ and $y$ is an $A|_{U}$-torsor on $\thesite/U$. A morphism ${(U,x) \to (V,y)}$ is a pair $(f,\lambda)$ where ${f: U \to V}$ is a morphism in $\thesite$ and $\lambda: f^{-1}y \to x$ is a morphism of $A|_{U}$-torsors on $\thesite/U$. The category $\torsorcat{A}$ is a stack in groupoids over $\thesite$ (See \cite[\href{https://stacks.math.columbia.edu/tag/04UJ}{Tag 04UJ}]{stacks-project}).
\end{sloppypar}
\begin{definition}\label{definition-abelian-gerbe-definition}
An abelian gerbe $\gerbe$ over $\xdown$ is a fibered category $\gerbe$ over $\thesite$ such that the following holds:
\begin{itemize}
\item $\gerbe$ is a stack in groupoids over $\thesite$.
\item For any object $U$ of $\thesite$, and $x,y \in \objects(\gerbe_{U})$, there exists a cover $\{U_i \to U\}$ of $U$ such that $x|_{U_i}, y|_{U_i}$ are isomorphic in $\gerbe_{U_i}$. 
\item For any object $U$ of $\thesite$, and $x \in \objects(\gerbe_{U})$, the sheaf of groups $\aut(x)$ on $\thesite/U$ is abelian.
\item For any object $U$ of $\thesite$, there exists a cover $\{U_i \to U\}$ of $U$ such that the fiber categories $\gerbe_{U_i}$ are non-empty. 
\end{itemize}
A morphism of abelian gerbes over $\xdown$ is a morphism of fibered categories over $\thesite$. We say $\gerbe$ is a pseudo abelian gerbe over $\xdown$ if it satisfies all the above conditions except possibly the last. Since we only consider gerbes which are abelian, we sometimes drop the adjective.  
\end{definition}
\begin{sloppypar}
Let $\gerbe$ be a gerbe on $\xdown$. Then there is an associated sheaf of abelian groups denoted by $\band(\gerbe)$ on $\thesite$ with the following property. For every object $U$ of $\thesite$, and $x \in \objects(\gerbe_{U})$, there is an isomorphism ${\band(\gerbe)|_U \to \aut(x)}$ such that for every morphism ${\psi : x \to y}$ in $\gerbe_{U}$, the following diagram commutes:
\end{sloppypar}
\[
  \begin{tikzcd}
    \band(\gerbe)(U) \arrow{d} \arrow{r}{=} & \band(\gerbe)(U) \arrow{d}  \\
    \aut(x) \arrow{r}{\alpha \to \psi \circ \alpha \circ \psi^{-1}} & \aut(y)
  \end{tikzcd}
\]
\begin{sloppypar}
  A morphism of abelian gerbes over $\xdown$ induces a group homomorphism of the associated bands. A gerbe banded by $A$ is an abelian gerbe $\gerbe$ over $X$ together with an isomorphism of sheaves ${\band(\gerbe) \to A}$. Let $\gerbe_1, \gerbe_2$ be two gerbes over $\xdown$ banded by $A_1$ and $A_2$ respectively. Suppose ${f: A_1 \to A_2}$ is a group homomorphism of sheaves. An $f\mbox{-equivariant}$ morphism ${\gerbe_1 \to \gerbe_2}$ is a morphism of gerbes $\gerbe_1 \to \gerbe_2$ such that the associated morphism of their bands fits in the following commutative diagram:
\end{sloppypar}
\[
  \begin{tikzcd}
    \band(\gerbe_1) \arrow{d} \arrow{r} & \band(\gerbe_2) \arrow{d}  \\
    A_1 \arrow{r}{f}  & A_2
  \end{tikzcd}
\]
An $A\mbox{-equivariant}$ (or just equivariant) morphism of gerbes banded by $A$ is a morphism which is equivariant with respect to the identity map of $A$. As in the case of torsors, equivariant morphisms of gerbes are isomorphisms. 
\begin{definition}\label{definition-hom-category-between-fibered-categories}
  \begin{sloppypar}
    If $Y_1, Y_2$ are categories fibered over a category $C$, then ${\hom_C(Y_1,Y_2)}$ is a category whose objects are morphisms $Y_1 \to Y_2$ of fibered categories and the morphisms are base preserving natural transformations. If $Y_1$ and $Y_2$ are fibered in groupoids over $C$, then the category $\hom_C(Y_1,Y_2)$ is a groupoid (See Chapter 3, \cite{olsson_algebraic_spaces_stacks_book}).
  \end{sloppypar}
\end{definition}
\begin{remark}\label{remark-iso-functors-induce-same-map-on-bands}
  Suppose $\gerbe_1$ and $\gerbe_2$ are abelian gerbes over $\xdown$ banded by $A_1$ and $A_2$, respectively. If $F,G: \gerbe_1 \to \gerbe_2$ are morphisms of gerbes which are isomorphic in the category $\hom_{\xdown}(\gerbe_1,\gerbe_2)$, then they induce identical map between their bands. 
\end{remark}
\begin{sloppypar}
Torsors and gerbes have functorial constructions in the sense of the following proposition:
\end{sloppypar}
\begin{proposition}\label{prop-hom-bands-gerbe-construction}
  \begin{sloppypar}
Let ${\psi: A_1 \to A_2}$ be a group homomorphism of abelian sheaves on $\xdown$. Suppose we are given an $A_1\mbox{-torsor}$ $y_1$ on $\xdown$ and an abelian gerbe $\gerbe_1$ on $\xdown$ which is banded by $A_1$. Then we have the following:
  \end{sloppypar}
  \begin{enumerate}
  \item There exists an $A_2\mbox{-torsor}$ $y_2$ and a $\psi\mbox{-equivariant}$ morphism ${f: y_1 \to y_2}$. Moreover, if $y_2'$ is another $\operatorname{A_2-torsor}$ with a $\psi\mbox{-equivariant} $ map ${f': y_1 \to y_2}$, then there is a unique $A_2\mbox{-equivariant}$ map ${h: y_2 \to y_2'}$ such that ${h \circ f = f'}$. 

\item
  \begin{sloppypar}
    There exists an abelian gerbe $\gerbe_2$ banded by $A_2$ and a $\psi\mbox{-equivariant}$ morphism of gerbes $F: \gerbe_1 \to \gerbe_2$.

    Suppose $\gerbe_2'$ is another gerbe banded by $A_2$ and admits a $\psi\mbox{-equivariant}$ morphism ${F': \gerbe_1 \to \gerbe_2'}$. Then there exists an $A_2\mbox{-equivariant}$ morphism ${H: \gerbe_2 \to \gerbe_2'}$ such that $H \circ F$ is isomorphic $F'$ in the category $\hom_{\thesite}(\gerbe_1,\gerbe_2')$. Furthermore, $H$ is unique up to an isomorphism in $\hom_{\thesite}(\gerbe_2,\gerbe_2')$.
  \end{sloppypar}
\end{enumerate}
\end{proposition}

\begin{sloppypar}
The above proposition also allows to define an abelian group structure on the equivariant isomorphism classes of $A\mbox{-torsors}$ and gerbes banded by $A$. Suppose $y$ and $y'$ are $A\mbox{-torsor}$ on $\xdown$. Then the product sheaf $y \times y'$ is an $(A \times A)\mbox{-torsors}$ on $\xdown$. We have the map ${ A \times A \to A}$ given by the group law. By the above proposition, there exists a unique (up to an isomorphism) $A\mbox{-torsor}$ $y \otimes y'$ with a morphism ${y \times y' \to y \otimes y'}$ that is equivariant with respect to the group law. It is easy to verify that the operation thus constructed defines an abelian group structure on the isomorphism classes of $A\mbox{-torsors}$. A similar construction can be made for gerbes banded by $A$.
\end{sloppypar}
\begin{theorem}\label{theorem-cohomologies-and-torsors-gerbes}
The sheaf cohomology group $\cohomology{1}{\xdown}{A}$ is isomorphic to the group of $A\mbox{-equivariant}$ isomorphism classes of $A\mbox{-torsors}$. Similarly, the sheaf cohomology group $\cohomology{2}{\xdown}{A}$ is isomorphic to the group of $A\mbox{-equivariant}$ isomorphism classes of abelian gerbes banded by $A$.
\end{theorem}

\section{More on Gerbes}
\begin{sloppypar}
The category of gerbes over $\xdown$ is a $2\mbox{-category}$ where $2\mbox{-morphisms}$ are as in Definition \ref{definition-hom-category-between-fibered-categories}. Suppose ${F,H: \gerbe_1 \to \gerbe_2}$ are morphisms of gerbes over $\xdown$. We say $F$ and $H$ are $2\mbox{-isomorphic}$ if they are isomorphic objects in the category $\hom_{\thesite}(\gerbe_1,\gerbe_2)$. 

The category $\torsorcat{A}$ is an abelian gerbe over $\xdown$ whose band is isomorphic to $A$. We refer to it as the trivial gerbe banded by $A$. We fix a choice of pullbacks for the fibered category  $\torsorcat{A}$ over $\thesite$. Suppose $y$ is an $A\mbox{-torsor}$ on $\xdown$ and $(U,z)$ is an object of $\torsorcat{A}$. Then the sheaf $\isom(y|_{U},z)$ is a torsor over the group sheaf ${\isom(y|_U,y|_U) \cong A|_{U}}$. The assignment ${(U,z) \to (U,\isom(y|_{U},z))}$ defines an $A\mbox{-equivariant}$ morphism of gerbes banded by $A$. We denote this morphism by $\isom(y,-)$. More generally, if $\gerbe$ is a gerbe banded by $A$ that admits a global section $y$, then the isom functor as defined above gives an $A\mbox{-equivariant}$ isomorphism of gerbes ${\gerbe \to \torsorcat{A}}$. Thus, as in the case of torsors, a gerbe which admits a global section is isomorphic to the trivial gerbe. 
\end{sloppypar}
\begin{remark}\label{remark_trivialzation_given_by_global_section_for_gerbe}
Any $A\mbox{-equivariant}$ morphism of gerbes ${\gerbe \to \torsorcat{A}}$ is $2\mbox{-isomorphic}$ to $\isom(y,-)$ for some global section $y$ of $\gerbe$.
\end{remark}

\begin{remark}\label{remark-composition-of-morphisms-and-addition-law-on-torsors}
  \begin{sloppypar}
    If $y_1$ and $y_2$ are two $A\mbox{-torsors}$ on $\xdown$ then the composite morphism 
$$\isom(y_1,-) \circ \isom(y_2,-) : \torsorcat{A} \to \torsorcat{A} $$
is $2\mbox{-isomorphic}$ to the morphism ${\isom(y_1 \otimes y_2,-)}$.
\end{sloppypar}
\end{remark}
\begin{lemma}\label{lemma-automorpism-of-functor-is-global-sections}
Let $\gerbe_1$ and $\gerbe_2$ be gerbes over $\xdown$ banded by $A_1$ and $A_2$, respectively. Suppose we have a morphism $ F: \gerbe_1 \to \gerbe_2 $
of gerbes. Then the automorphism group of the object $F$ in the category $\hom_{\xdown_{\tau}}(\gerbe_1,\gerbe_2)$ is isomorphic to the global sections $A_2(\xdown)$ of the band of $\gerbe_2$. Suppose we use overline to denote this isomorphism. Then for any $2\mbox{-morphisms}$ ${\alpha: F \to F}$ and ${\gamma: G \to F}$, we have $ \overline{\alpha} = \overline{\gamma \circ \alpha \circ \gamma^{-1}}$.
\end{lemma}
\begin{proof}
  \begin{sloppypar}
    Suppose ${\{U_i \to \xdown \}_{i \in I}}$ is a cover of $\xdown$ such that the fiber categories ${{\gerbe_1}_{U_i}}$ are non-empty. Choose ${x_i \in \objects({\gerbe_1}_{U_i})}$. A base preserving natural transformation from $F$ to itself gives an automorphism $F(x_i) \to F(x_i)$ in the fiber category ${\gerbe_2}_{U_i}$ for each $i \in I$. Since $\aut(F(x_i))|_{U_i}$ is isomorphic to ${A_2}|_{U_i}$, we get a section $a_i$ of the sheaf $A_2$ over $U_i$. For a pair of indices $i,j \in I$, there exists a cover ${\{ U_k \to U_i \times_{\xdown} U_J \}_{k \in I_{i,j}}}$ of $U_i \times_{\xdown} U_j$ where the objects $x_i, x_j$ are isomorphic. For $k \in I_{i,j}$, after choosing an isomorphism $\beta: x_i|_{U_k} \to x_j|_{U_k}$ in the category ${\gerbe_1}_{U_k}$, we get the following commutative diagram of arrows in the category ${\gerbe_2}_{U_k}$:
  \end{sloppypar}
  \[
  \begin{tikzcd}
    F(x_i|_{U_k}) \arrow{r}{a_i|_{U_k}} \arrow{d}[swap]{F(\beta)} & F(x_i|_{U_k}) \arrow{d}{F(\beta)} \\
    F(x_j|_{U_k}) \arrow{r}{a_j|_{U_k}}  & F(x_j|_{U_k})
  \end{tikzcd}
\]
\begin{sloppypar}
  Therefore we conclude that $a_i|_{U_k} = a_j|_{U_k}$ for each $k \in I_{i,j}$. By the sheaf condition for the cover ${\{ U_k \to U_i \times_{\xdown} U_J \}_{k \in I_{i,j}}}$ we see that ${a_i|_{U_i \times_{\xdown} U_J} = a_j|_{U_i \times_{\xdown} U_J}}$ for every $i,j \in I$. Using the sheaf condition again but for the cover ${\{U_i \to \xdown \}_{i \in I}}$ we conclude that there exists a global section $a \in A_2(\xdown)$ whose restriction to $U_i$ is $a_i$ for each $i \in I$. In other words, a base preserving transformation from $F$ to itself is given by the action of a global section of $A_2$. The second part of the lemma follows easily from this observation.
\end{sloppypar}
\end{proof}
\begin{lemma}\label{lemma-G-action-on-2-morphisms-same-as-on-global-sections}
Let $\gerbe_1,\gerbe_2,$ and $\gerbe_3$ be abelian gerbes on $\xdown$ banded by $A_1,A_2,$ and $A_3$ respectively. Let ${G: \gerbe_1 \to \gerbe_2}$ and ${F: \gerbe_2 \to \gerbe_3}$ be morphisms of gerbes. Let ${\delta_F: A_2 \to A_3}$ denote the map induced by $F$ on the bands of the gerbes. Let 
${\alpha: F \to F}$ and ${\beta: G \to G}$
be 2-morphisms. Then with the notation as in Lemma  \ref{lemma-automorpism-of-functor-is-global-sections}, we have
${\overline{\alpha \star \beta} = \delta_F(\overline{\beta}) \cdot \overline{\alpha}}$.
\end{lemma}
\begin{proof}
  \begin{sloppypar}
    For $U \in \objects(\thesite)$ and $x \in \objects({\gerbe_1}_{U_i})$, the $2\mbox{-morphism}$ ${ \alpha \star \beta: F \circ G \to F \circ G }$ is locally given by the commutative diagram below:  
  \end{sloppypar}
  \[
  \begin{tikzcd}
    F(G(x)) \arrow{r}{F(\beta)} \arrow{dr}[swap]{\alpha * \beta} & F(G(x)) \arrow{d}{\alpha} \\
    \empty & F(G(x))
  \end{tikzcd}
\]
It follows from the proof of Lemma \ref{lemma-automorpism-of-functor-is-global-sections} that this corresponds to the element $ \delta_F(\overline{\beta}) \cdot \overline{\alpha}$.
\end{proof}

\section{Group actions}
{
\newcommand{\trivialtorsor}{y^{triv}}

We use additive notation for the group structure of $A$. Let $G$ be a finite group. By a (left) action of the group $G$ on $A$ we mean a group homomorphism ${ G^{opp} \to \isom_{Ab(\thesite)}(A,A)}$.
\begin{definition}\label{def-G-action-on-A-and-invariants}
Suppose the group $G$ acts on the abelian sheaf $A$ by automorphisms ${\{\rho_g\}_{g \in G}}$. The invariants of $A$ is the sheaf defined by the assignment ${ U \to A(U)^{G}}$. We denote this sheaf by $\invariants$. We have a map ${i: A^G \to A }$ such that for every $g \in G$, we have ${ \rho_{g} \circ i = i }$.
\end{definition}
\begin{definition}\label{def-group-cohomlogies-vanish-locally}
  \begin{sloppypar}
    Suppose we have an action of the group $G$ on $A$. For an index $j > 0$ we say that ``the $j^{th}$ group cohomology locally vanishes'' if the following holds: For every object $U$ of $\thesite$, there exists a cover ${ \{U_i \to U\}_{i \in I}}$ such that for all $i \in I$, the group cohomology ${\cohomology{j}{G}{A(U_i)}}$ is zero. 

    We say that ``the $j^{th}$ group cohomology classes locally vanish'' if the following holds: For every object $U$ of $\thesite$, and a group cohomology class $\alpha \in \cohomology{j}{G}{A(U)}$, there exists a cover ${\{U_i \to U\}_{i \in I}}$ such that for all $i \in I$, the restriction of $\alpha$ to ${\cohomology{j}{G}{A(U_i)}}$ is zero. It is clear the latter condition is weaker than the former.
  \end{sloppypar}
\end{definition}
\begin{definition}\label{def-induced-torsor-and-gerbe}
Suppose we have a group action by $G$ on the abelian sheaf $A$. Let $\invariants \xrightarrow{i} A$ denote the sheaf of invariants. We say an $A\mbox{-torsor}$ $y$ is induced from the invariants if there exists an $\invariants\mbox{-torsor}$ $\overline{y}$ and an $i\mbox{-equivariant}$ morphism $ \overline{y} \to y$.

Suppose $\gerbe$ is an abelian gerbe banded by $A$. We say $\gerbe$ is induced from the invariants if there exists a gerbe $\overline{\gerbe}$ banded by $\invariants$ and an $i\mbox{-equivariant}$ morphism of gerbes 
${\overline{\gerbe} \to \gerbe}$.
\end{definition}
\begin{definition}\label{definition-G-action-on-torsor}
Let $y$ be an $A\mbox{-torsor}$ on $\xdown$. By an action of $G$ on $y$ we mean an assignment of a morphism ${\lambda_g: y \to y }$ for every $g \in G$ 
which is equivariant with respect to an automorphism ${\phi_g: A \to A}$ of $A$. Moreover, we require that the automorphisms $\{\phi_g\}_{g \in G}$ define a (left) $G\mbox{-action}$ on $A$. We refer to the action on $A$ as the underlying action. Given the underlying action on $A$, a lift of the action on $y$ is a $G\mbox{-action}$ on $y$ whose underlying action on $A$ is the given action.
\end{definition}
\begin{definition}\label{def-G-equivariant-morphism-of-torsors}
  \begin{sloppypar}
    Let ${f: A_1 \to A_2}$ be a morphism of abelian sheaves on $\xdown$. Let $y_1$ (resp. $y_2$) be an $A_1\mbox{-torsor}$ (resp. $A_2\mbox{-torsor}$) on $\xdown$. Suppose we have an action of $G$ on $y_1$ given by morphisms ${\{\lambda_g^1\}_{g \in G}}$ and also on $y_2$ given by morphisms ${\{\lambda_g^2\}_{g \in G}}$. By a $G\mbox{-equivariant}$ morphism ${\lambda: y_1 \to y_2}$ we mean a map of torsors ${\lambda: y_1 \to y_2}$ such that it is $f\mbox{-equivariant}$ and for $g \in G$, the following diagram commutes:
  \end{sloppypar}
  \[
  \begin{tikzcd}
    y_1 \arrow{d}{\lambda_g^1} \arrow{r}{\lambda} & y_2 \arrow{d}{\lambda_g^2} \\
    y_1 \arrow{r}{\lambda} & y_2
  \end{tikzcd}
\]
It is easy to see that if such a $G\mbox{-equivariant}$ morphisms exists, then $f$ is a $G\mbox{-equivariant}$ map of sheaves with respect to the underlying actions on $A_1$ and $A_2$. 
\end{definition}

\begin{definition}\label{def-2-cochain-for-torsor-action}
Suppose $y$ is an $A\mbox{-torsor}$ with an action of $G$ given by morphisms $\{\lambda_g\}_{g \in G}$ as in Definition \ref{definition-G-action-on-torsor}. Then for  $g,h \in G$, we denote by $\chi_{\lambda}(g,h) \in A(\xdown)$ the scalar that corresponds to the unique $A\mbox{-equivariant}$ map which makes the diagram below commute. 
\[
  \begin{tikzcd}
    y \arrow{d}{\lambda_h} \arrow{r}{\lambda_{gh}} & y \arrow{d}{\chi_{\lambda}(g,h)}  \\
    y \arrow{r}{\lambda_g} & y 
  \end{tikzcd}
\]
We regard $\chi_{\lambda}$ as a $2\mbox{-cochain}$ with values in $A(\xdown)$.
\end{definition}
We want to prove that in the situation of Definition \ref{def-2-cochain-for-torsor-action}, the cochain $\chi_\lambda$ is a cocycle with respect to the underlying action on $A$. Moreover, its cohomology class in $\cohomology{2}{G}{A(\xdown)}$ is independent of the chosen lift of the $G\mbox{-action}$ on $y$.
\begin{lemma}\label{lemma-change-in-lifted-action-on-torsors-coboundary}
  \begin{sloppypar}
    Let $y$ be an $A\mbox{-torsor}$. Suppose we have two $G\mbox{-actions}$ on $y$ given by morphisms ${\{\lambda_g^1\}_{g \in G}}$ and ${\{\lambda_g^2\}_{g \in G}}$ such that their underlying action on $A$ is identical. Then the two cochains ${\chi_{\lambda^1}}$ and ${\chi_{\lambda^2}}$ differ by a coboundary. Conversely, given any $2\mbox{-cochain}$ $\chi$ which is cohomologous to $\chi_{\lambda^1}$, there exists a lift of the action of $G$ such that the corresponding cochain is $\chi$. 
  \end{sloppypar}
\end{lemma}
\begin{proof}
  \begin{sloppypar}
    It follows from Proposition \ref{prop-hom-bands-gerbe-construction} that for $g \in G$, there exists an $A\mbox{-equivariant}$ map given by $a_g \in A(\xdown)$ such that $ \lambda^2_g = a_g \circ \lambda^1_g $. Then for $g,h \in G$, we see that the ${\chi_{\lambda^1}(g,h)}$ and ${\chi_{\lambda^2}(g,h)}$ differ by  ${(g \cdot a_h)  - a_{gh} + a_g }$ which is the coboundary of the $1\mbox{-cochain}$ $\{a_g\}_{g \in G}$. The second part of the lemma follows easily by suitably altering the action on $y$.
  \end{sloppypar}
\end{proof}
\begin{lemma}\label{lemma-functoriality-of-2-cochain-for-torsors}
  \begin{sloppypar}
    Suppose we have a $G\mbox{-equivariant}$ morphism of torsors ${\lambda: y_1 \to y_2 }$ as in Definition \ref{def-G-equivariant-morphism-of-torsors}. Then we have $ \chi_{\lambda^2} = f \circ \chi_{\lambda^1}$.
  \end{sloppypar}
\end{lemma}
\begin{proof}
For $g,h \in G$, we have 
\begin{align*}
  \lambda \circ \left( \chi_{\lambda^1}(g,h) \circ \lambda^1_{gh} \right) &= \lambda \circ \left(\lambda^1_g \circ \lambda^1_h \right) \\
  &= \lambda^2_g \circ \lambda^2_h \circ \lambda \\
  &= \left( \chi_{\lambda^2}(g,h) \circ \lambda^2_{gh} \right) \circ \lambda \\
  &= \chi_{\lambda^2}(g,h) \circ \lambda \circ \lambda^1_{gh}
\end{align*}
Since $\lambda^1_{gh}$ is an isomorphism, we conclude that the diagram below commutes.
\[
  \begin{tikzcd}[column sep = huge, row sep = large]
    y_1 \arrow{d}{\chi_{\lambda^1}(g,h)} \arrow{r}{\lambda} & y_2 \arrow{d}{\chi_{\lambda^2}(g,h)} \\
    y_1 \arrow{r}{\lambda} & y_2
  \end{tikzcd}
\]
Therefore, it follows that we have $ f(\chi_{\lambda^1}(g,h)) = \chi_{\lambda^2}(g,h)$.
\end{proof}
\begin{situation}\label{situation-G-action-on-torsor-with-local-section}
  \begin{sloppypar}
    Suppose we have an action of $G$ on an $A\mbox{-torsor}$ $y$ as in Definition \ref{definition-G-action-on-torsor}. Suppose $U \in \objects(\thesite)$ is such that $y$ admits a section $ x \in y(U)$. For $g \in G$, we set $c_g \in A(U)$ to be the unique scalar such that ${\lambda_g(x) = c_g \cdot x}$. 
  \end{sloppypar}
\end{situation}
\begin{lemma}\label{lemma-local-section-commutative-diagram-for-torsors}
In the above Situation \ref{situation-G-action-on-torsor-with-local-section}, we have the following:
\begin{enumerate}
  \begin{sloppypar}
  \item The $1\mbox{-cochain}$ $\{c_g\}_{g \in G}$ with values in $A(U)$ is independent of the chosen section $x \in y(U)$ up to a coboundary. 
  \item The restriction of $\chi_{\lambda}$ to $U$ is the coboundary of the $1\mbox{-cochain}$ $\{c_g\}_{g \in G}$. 
  \end{sloppypar}
\end{enumerate}
\end{lemma}
\begin{proof}
  For $c \in A(U)$, we have
  \begin{align*}
   \lambda_g(c \cdot x ) &= (g \cdot c) \cdot (\lambda_g(x)) \\
                         &= \left( g \cdot c + c_g \right) \cdot x \\
                         &= \left(g \cdot c  + c_g  - c \right) \cdot (c \cdot x)
  \end{align*}
This proves first part of the lemma.
For $g,h \in G$, we see that
\begin{align*}
  \left(\lambda_{g} \circ \lambda_{h}\right)(x) &= \lambda_{g}\left( c_h \cdot x \right) \\
                                   &= \left(c_g + (g \cdot c_h) \right) \cdot x 
\end{align*}
On the other hand, 
\begin{align*}
  \left(\lambda_{g} \circ \lambda_h\right) (x) &= \left(\chi_{\lambda}(g,h) \circ \lambda_{gh} \right)(x) \\
                                  &= (\chi_{\lambda}(g,h) + c_{gh}) \cdot   x
\end{align*}
Thus we see that 
$$ \chi_{\lambda}(g,h) + c_{gh} = c_g + (g \cdot c_h) .$$
In other words, $\chi_{\lambda}(g,h)$ restricted to $U$ is the
coboundary of $\{c_g\}_{g \in G}$. 
\end{proof}
\begin{proposition}\label{prop-chi-lambda-is-cocycle}
Suppose we have an action of $G$ on an $A\mbox{-torsor}$ $y$ as in Definition \ref{definition-G-action-on-torsor}. Then $\chi_{\lambda}$ is a $2\mbox{-cocycle}$. 
\end{proposition}
\begin{proof}
There exists a cover ${\{U_i \to \xdown \}_{i \in I}}$ of $\xdown$ such that $y$ admits section on each $U_i$. By Lemma \ref{lemma-local-section-commutative-diagram-for-torsors}, we see that
the restriction of $\chi_{\lambda}$ on each $U_i$ is a coboundary and therefore its coboundary vanishes on each $U_i$. Thus we see that $\chi_{\lambda}$ is a cocycle. 
\end{proof}
Suppose $y$ is an $A\mbox{-torsor}$ and we have an action $G$ on $y$ given by morphisms ${\{\lambda_g\}_{g \in G}}$. For an object $U$ of $\thesite$, we set 
$$ y_{\lambda}(U) = \left\{x \in y(U) \middle|  \lambda_g(x) = x \text{ for all } g \in G \right\}.$$ 
It is evident that $y_{\lambda}$ is a pseudo $\invariants\mbox{-torsor}$. 
\begin{lemma}\label{lemma-local-cohomology-obstruction-for-local-sections-of-y-lambda}
Suppose the $2\mbox{-cochain}$ $\chi_{\lambda}$ is identically zero. Let $U$ be an object of $\thesite$ over which $y$ admits a section. Then there is a cohomology class in $\cohomology{1}{G}{A(U)}$ which is zero if and only if $y_{\lambda}(U)$ is non-empty. 
\end{lemma}
\begin{proof}
  We choose a section in $x \in y(U)$ and define the $1\mbox{-cochain}$ $\{c_g\}_{g \in G}$ as in Situation \ref{situation-G-action-on-torsor-with-local-section}. If $\chi_{\lambda}$ is identically zero, then by Lemma \ref{lemma-local-section-commutative-diagram-for-torsors} we see that the cochain $\{c_g\}_{g \in G}$ is a cocycle whose cohomology class is independent of the chosen section $x$. It also follows from Lemma \ref{lemma-local-section-commutative-diagram-for-torsors} that a section in $y(U)$ is fixed by the morphisms $\{\lambda_g\}_{g \in G}$ if and only if this cohomology class is zero. 
\end{proof}
\begin{lemma}\label{lemma-y-lambda-admits-local-sections-iff-general-case}
If $y_\lambda$ is a legitimate $A^G\mbox{-torsor}$, then
the $2\mbox{-cochain}$ $\chi_{\lambda}$ is identically zero. The
converse holds if the first group cohomology classes for the underlying action on $A$ locally vanish (Definition \ref{def-group-cohomlogies-vanish-locally}).
\end{lemma}
\begin{proof}
  \begin{sloppypar}
    Suppose for an object $U$ of $\xdown_{\tau}$, the set $y_{\lambda}(U)$ is non-empty. Then we have a section $x \in y(U)$ such that for all $g \in G$, we have ${\lambda_{g}(x) = x}$. Then (See Lemma \ref{lemma-local-section-commutative-diagram-for-torsors}) we see that the restriction of $\chi_{\lambda}$ to $U$ is zero. Thus, if $y_{\lambda}$ admits local sections, then the cochain $\chi_{\lambda}$ vanishes everywhere. 

    Conversely, suppose $\chi_{\lambda}$ is identically zero. Since $y$ admits local sections, there exists a cover ${\{U_i \to \xdown\}_{i \in I}}$ such that $y(U_i)$ is non-empty for each $i \in I$. By Lemma \ref{lemma-local-cohomology-obstruction-for-local-sections-of-y-lambda}, there is a cohomology class in $\cohomology{1}{G}{A(U_i)}$ which is zero if and only if $y_\lambda(U_i)$ is non-empty. If first cohomology classes locally vanish then there exists a refinement of this cover over which $y_{\lambda}$ admits sections.
  \end{sloppypar}
\end{proof}
\begin{sloppypar}
  Now we shift our attention to gerbes. We use the following terminology. Suppose $F,H : \gerbe_1 \to \gerbe_2 $ are morphisms of gerbes. We say $F$ and $H$ 2-commute if they are isomorphic in the category ${\hom_{\xdown}(\gerbe_1,\gerbe_2)}$. Similarly, we say a diagram of 1-morphisms 2-commutes, if the arrows with the same source and the target are 2-isomorphic. 

Let $\gerbe$ be an abelian gerbe banded by $A$.
\end{sloppypar}
\begin{definition}\label{def-group-action-on-gerbe}
An action of $G$ on $\gerbe$ is an assignment of a 1-morphism ${F_g: \gerbe \to \gerbe}$ for every $g \in G$ such that the following holds:
\begin{itemize}
\item The morphism associated to the identity is $A\mbox{-equivariant}$. 
\item For every $g,h \in G$, the morphisms $F_{g \cdot h}$ and $F_{g} \circ F_{h}$ 2-commute.
\end{itemize}
\begin{sloppypar}
  By ``connecting 2-morphisms'' we mean a choice of 2-morphisms ${\alpha_{g,h}: F_{g \cdot h} \to F_g \circ F_h }$ for every $g,h \in G$. If we have a $G\mbox{-action}$ and connecting 2-morphisms as above, then we denote this data by the pair $(F,\alpha)$. Given the action on $\gerbe$, we have an action of $G$ on $A$ such that for all $g \in G$, the morphism $F_g$ is equivariant with respect to the action of $g$ on $A$. We refer to this action as the underlying action of $G$.  
\end{sloppypar}
\end{definition}
\begin{definition}\label{def-3-cochain-def-with-group-action-and-choice}
  \begin{sloppypar}
Suppose $(F,\alpha)$ is the data of $G$ action on the gerbe $\gerbe$ together with connecting $2\mbox{-morphisms}$. Then for $g_1, g_2, g_3 \in G$ we have two 2-morphisms ${F_{g_1 \circ g_2 \circ g_3} \rightrightarrows F_{g_1} \circ F_{g_2} \circ F_{g_3}}$ and we define $\kappa_{(F,\alpha)}(g_1,g_2,g_3)$ to be the ``difference between the two'' as shown in the diagram below:
  \end{sloppypar}
  \[
    \begin{tikzcd}[column sep = huge]
      F_{g_1 \cdot g_2 \cdot g_3}  \arrow{d}{\kappa_{(F,\alpha)}(g_1,g_2,g_3)} \arrow{r}{\alpha_{g_1,g_2g_3}} & F_{g_1} \circ F_{g_2 \cdot g_3}  \arrow{r}{id_{F_{g_1}} \star \, \alpha_{g_2,g_3}} &[2em] F_{g_1} \circ  F_{g_2} \circ F_{g_3} \arrow{d}{\rotatebox{90}{=}}\\
      F_{g_1 \cdot g_2 \cdot g_3} \arrow{r}{\alpha_{g_1g_2,g_3}}  & F_{g_1 \cdot g_2} \circ F_{g_3} \arrow{r}{\alpha_{g_1,g_2} \star \, id_{F_{g_3}}}  &  F_{g_1} \circ F_{g_2}  \circ F _{g_3}
    \end{tikzcd}
  \]
In this context, we regard $\kappa_{(F,\alpha)}$ as a $3\mbox{-}cochain$ with values in $A(\xdown)$ (See Lemma \ref{lemma-automorpism-of-functor-is-global-sections}).
\end{definition}
We want to prove that the 3-cochain $\kappa_{(F,\alpha)}$ is a cocycle whose cohomology class in $\cohomology{3}{G}{A(\xdown)}$ is independent of the choice of the connecting 2-morphisms. We follow a strategy analogous to that for torsors. 
\begin{lemma}\label{lemma-different-connecting-changes-3-cocycle-by-a-coboundary}
  \begin{sloppypar}
 Suppose we have a $G\mbox{-action}$ on $\gerbe$ given by morphisms ${\{F_g\}_{g \in G}}$. Suppose for $g,h \in G$, we have two connecting $2\mbox{-morphisms}$ $\alpha^1_{g,h}$ and $\alpha^2_{g,h}$. Then the 3-cochains $\kappa_{(F,\alpha^2)}$ and $\kappa_{(F,\alpha^1)}$ differ by a coboundary. Moreover, if $\kappa$ is a $3\mbox{-cochain}$ that is cohomologous to $\kappa_{F,\alpha^1}$, then there exists a choice of connecting $2\mbox{-morphisms}$ such that the associated cochain is $\kappa$. 
  \end{sloppypar}
\end{lemma}
\begin{proof}
\begin{sloppypar}
  For $g,h \in G$, there exists an automorphism of $F_{g \cdot h}$ corresponding to an element ${c(g,h) \in A(\xdown)}$ such that the following diagram of arrows in $\hom_{\thesite}(\gerbe,\gerbe)$ commutes:
\end{sloppypar}
\[
  \begin{tikzcd}
    F_{g \circ h} \arrow{r}{\alpha^1_{g,h}} \arrow{d}{c(g,h)} & F_{g} \circ F_{h}  \\
    F_{g \circ h} \arrow{ur}[swap]{\alpha^2_{g,h}}
  \end{tikzcd}
\]
We write $\alpha^1_{g,h}$ as $\alpha^2_{g,h} \circ c(g,h)$ in Definition \ref{def-3-cochain-def-with-group-action-and-choice} and see that the difference between $\kappa_{(F,\alpha^1)}(g_1,g_2,g_3)$ and $\kappa_{(F,\alpha^2)}(g_1,g_2,g_3)$ is given by (See Lemma \ref{lemma-G-action-on-2-morphisms-same-as-on-global-sections})
$$ c(g_1,g_2g_3) + (g_1 \cdot c(g_2,g_3)) - c(g_1,g_2) - c(g_1g_2,g_3).$$
This is the coboundary of the 2-cochain $\{c(g,h)\}_{g,h \in G}$. To prove the second part of the lemma, we note that we can suitably modify the $2\mbox{-morphisms}$ such that the resulting cochain is altered by any given coboundary. 
\end{proof}
Thus we see that up to a coboundary the $3\mbox{-cochain}$ $\kappa_{(F,\alpha)}$ is independent of the choice of connecting $2\mbox{-morphisms}$. We denote by $\kappa_F$ the equivalence class of $\kappa_{(F,\alpha)}$ modulo coboundaries.
\begin{definition}\label{def-G-equivariant-morphism-of-gerbes}
Suppose $\gerbe_1$ and $\gerbe_2$ are abelian gerbes over $\xdown$ banded by $A_1$ and $A_2$, respectively. Suppose we have a $G\mbox{-action}$ on both given by morphisms ${\{F^1_g\}_{g \in G}}$ and ${\{F^2_g\}_{g \in G}}$. By a $G\mbox{-equivariant}$ morphisms of gerbes we mean a morphism of gerbes ${F: \gerbe_1 \to \gerbe_2}$ such that for $g \in G$, the following diagram 2-commutes.
\[
  \begin{tikzcd}
    \gerbe_1 \arrow{d}{F^1_g} \arrow{r}{F} & \gerbe_2 \arrow{d}{F^2_g} \\
    \gerbe_1 \arrow{r}{F} & \gerbe_2 
  \end{tikzcd}
\]
In this situation, we have an underlying action of $G$ on both $A_1$ and $A_2$ while the map induced by $F$ gives a $G\mbox{-equivariant}$ map from $A_1$ to $A_2$.
\end{definition}
\begin{proposition}\label{prop-g-equivariant-map-and-3-cochain}
  \begin{sloppypar}
    Suppose in the situation of Definition \ref{def-G-equivariant-morphism-of-gerbes} we denote by ${\delta_{F}: A_1 \to A_2}$ the map induced by $F$ on the bands. Then we have ${\kappa_{F^2} = \delta_{F} \circ \kappa_{F^1}}$.
  \end{sloppypar}
\end{proposition}
\begin{proof}
  \begin{sloppypar}
Let ${\{\alpha^1_{g,h}\}_{g,h \in G}}$ be a choice of connecting $2\mbox{-morphisms}$ for the group action on $\gerbe_1$. For $g \in G$, we choose a $2\mbox{-morphism}$ ${\epsilon_g: F \circ F^1_g \to F^2_g \circ F}$. Then for $g,h \in G$, we denote by $\epsilon(g,h)$ the composite $2\mbox{-morphism}$ below:
  \end{sloppypar}
  \[
  \begin{tikzcd}[column sep = huge]
    F \circ F^1_g \circ F^1_h \arrow{r}{\epsilon_g \star \, id_{F^1_h}} & F^2_g \circ F \circ F^1_h \arrow{r}{id_{F^2_g} \star \, \epsilon_h} &  F^2_g \circ F^2_h  \circ F 
  \end{tikzcd}
\]
We can verify that for $g_1,g_2,g_3 \in G$, the following diagram commutes:
\[
\begin{tikzcd}[column sep = huge]
  F \circ F^1_{g_1} \circ F^1_{g_2} \circ F^1_{g_3} \arrow{d}{=} \arrow{r}{\epsilon(g_1,g_2) \, \star id_{F^1_{g_3}}} & F^2_{g_1} \circ F^2_{g_2} \circ F \circ F^1_{g_3} \arrow{r}{id_{-} \star \, \epsilon_{g_3}} & F^2_{g_1} \circ F^2_{g_2} \circ F^2_{g_3} \circ F \arrow{d}{=} \\
 F \circ F^1_{g_1} \circ F^1_{g_2} \circ F^1_{g_3} \arrow{r}{\epsilon_{g_1} \, \star  id_{-}} & F^2_{g_1} \circ F \circ F^1_{g_2} \circ F^1_{g_3} \arrow{r}{id_{F^2_{g_1}} \star \, \epsilon(g_2,g_3)} & F^2_{g_1} \circ F^2_{g_2} \circ F^2_{g_3} \circ F
\end{tikzcd}
\]
From Lemma \ref{lemma-G-action-on-2-morphisms-same-as-on-global-sections} and Lemma \ref{lemma-automorpism-of-functor-is-global-sections} we can argue that there exists a unique $2\mbox{-morphism}$ ${\alpha^2_{g,h}: F^2_{gh} \to F^2_g \circ F^2_h}$ such that the diagram below commutes:
\[
  \begin{tikzcd}[column sep = huge]
    F^1_{g \circ h} \circ F \arrow{d}{\epsilon_{gh}} \arrow{r}{\alpha^1_{g,h} \, \star id_F} & F^1_g \circ F^1_h \circ F \arrow{d}{\epsilon(g,h)} \\
    F \circ F^2_{g \circ h} \arrow{r}{id_F \star \, \alpha^2_{g,h}} & F \circ F^2_g \circ F^2_h  
  \end{tikzcd}
\]
Now we can verify that the diagram below commutes with the above choice of connecting $2\mbox{-morphisms}$:  
\[
  \begin{tikzcd}[row sep = large, column sep = huge]
    F^1_{g_1g_2g_3} \circ F \arrow{d}{\epsilon_{g_1g_2g_3}} \arrow{r}{\kappa_{(F^1,\alpha^1)} \star id_F} & F^1_{g_1g_2g_3} \circ F \arrow{d}{\epsilon_{g_1g_2g_3}} \\
    F \circ F^2_{g_1g_2g_3}  \arrow{r}{ id_F \star \, \kappa_{(F^2,\alpha^1)}} & F \circ F^2_{g_1g_2g_3} 
  \end{tikzcd}
\]
It follows that we have (See Lemma \ref{lemma-G-action-on-2-morphisms-same-as-on-global-sections} and Lemma \ref{lemma-automorpism-of-functor-is-global-sections})
  $${ \delta_{F}(\kappa_{F^1,\alpha^1}(g_1,g_2,g_3)) = \kappa_{F^1,\alpha^1}(g_1,g_2,g_3)}.$$
\end{proof}
\begin{corollary}\label{cor-3-cochain-depends-only-on-2-iso-class-upto-coboundary}
Suppose we have two $G\mbox{-actions}$ on the gerbe $\gerbe$ given by morphisms ${\{F_g\}_{g \in G}}$ and $\{F'_g\}_{g \in G}$ such that for $g \in G$, the morphisms $F_g$ and $F'_g$ are $2\mbox{-isomorphic}$. Then we have 
${\kappa_{F} = \kappa_{F'}}$.
\end{corollary}
\begin{proof}
In this case, the identity morphism of ${\gerbe}$ is $G\mbox{-equivariant}$ with respect to the two $G\mbox{-actions}$. Therefore the claim follows from Proposition \ref{prop-g-equivariant-map-and-3-cochain}.
\end{proof}
The following is the analogue of the Situation \ref{situation-G-action-on-torsor-with-local-section}, where we prove that the $3\mbox{-cochain}$ $\kappa_{(F,\alpha)}$ is locally a coboundary. 
\begin{situation}\label{situation-for-local-objects-for-gerbe-F-alpha}
Let $(F,\alpha)$ be the data of a group action by $G$ on a gerbe $\gerbe$ banded by $A$. Let $U$ be an object of $\thesite$ such that the following holds:
\begin{itemize}
\item There is an object $y \in \objects(\gerbe_U)$ over $U$. 
\item For all $g \in G$, $F_g(y)$ is isomorphic to $y$ in $\gerbe_U$.
\end{itemize}
We choose isomorphisms ${ \phi_g: y \to F_g(y)}$ for all $g \in G$. For $g,h \in G$, we define $\gamma(g,h)$ to be the unique isomorphism in $\gerbe_U$ which makes the following diagram commute:
\[
  \begin{tikzcd}
    y \arrow{d}{\phi_g} \arrow{r}{\phi_{gh}} & F_{gh}(y) \arrow{d}{\gamma(g,h)} \\
    F_g(y) \arrow{r}{F_g(\phi_h)} & \left(F_g \circ F_h\right) (y) 
  \end{tikzcd}
\]
We set scalars $c(g,h) \in A(U)$ such that ${c(g,h) \cdot \alpha_{g,h}(y) = \gamma(g,h)}$. 
\end{situation}
\begin{lemma}\label{lemma-commutative-diagram-for-local-objects-gerbe-F-alpha}
  In the above Situation \ref{situation-for-local-objects-for-gerbe-F-alpha} for $g_1,g_2,g_3 \in G$, let $\omega(g_1,g_2,g_3)$ be the unique isomorphism in $\gerbe_U$ that fits in the following commutative diagram:
\[
  \begin{tikzcd}[column sep = huge]
    F_{g_1g_2}(y) \arrow{d}{\gamma(g_1,g_2)} \arrow{r}{F_{g_1g_2}(\phi_{g_3})} & ( F_{g_1g_2} \circ F_{g_3})(y) \arrow{d}{\omega(g_1,g_2,g_3)} \\
    F_{g_1} \circ F_{g_2} (y) \arrow{r}{F_{g_1} \circ F_{g_2}(\phi_{g_3})} &     ( F_{g_1} \circ F_{g_2}\circ F_{g_3} )(y)
  \end{tikzcd}
\]
Then we claim that the diagram below commutes:
\[
  \begin{tikzcd}[column sep = huge]
    F_{g_1g_2g_3}(y) \arrow{d}{\rotatebox{90}{=}} \arrow{r}{\gamma(g_1,g_2g_3)} & (F_{g_1} \circ F_{g_2g_3})(y) \arrow{r}{F_{g_1}(\gamma(g_2,g_3))} & (F_{g_1} \circ F_{g_2} \circ F_{g_3}) (y) \arrow{d}{\rotatebox{90}{=}} \\
    F_{g_1g_2g_3}(y) \arrow{r}{\gamma(g_1g_2,g_3)} & (F_{g_1g_2} \circ F_{g_3})(y) \arrow{r}{\omega(g_1,g_2,g_3)} & (F_{g_1} \circ F_{g_2} \circ F_{g_3}) (y) 
  \end{tikzcd}
\]
\end{lemma}
\begin{proof}
  \begin{sloppypar}
    We have two morphisms ${y \rightrightarrows (F_{g_1} \circ F_{g_2} \circ F_{g_3}) (y)}$ given by 
$${F_{g_1}(\gamma(g_2,g_3)) \circ \, \gamma(g_1,g_2g_3) \circ \, \phi_{g_1g_2g_3}} \text{ and } {\omega(g_1,g_2,g_3) \circ \, \gamma(g_1g_2,g_3) \circ \, \phi_{g_1g_2g_3}}.$$ 
It is straightforward to verify that both of these morphisms are equal to the morphism below:
  \end{sloppypar}
\[
  \begin{tikzcd}
    y \arrow{r}{F_{g_1}} & F_{g_1}(y) \arrow{r}{F_{g_1}(\phi_{g_2})} & (F_{g_1} \circ F_{g_2}) (y) \arrow{rr}{(F_{g_1} \circ F_{g_2})(\phi_{g_3})} & \empty & (F_{g_1} \circ F_{g_2} \circ F_{g_3}) (y)
  \end{tikzcd}
\]
Since $\phi_{g_1g_2g_3}$ is an isomorphism, the conclusion follows. 
\end{proof}
\begin{lemma}\label{lemma-locally-kappa-f-alpha-is-coboundary}
In Situation \ref{situation-for-local-objects-for-gerbe-F-alpha} we have the following:
\begin{enumerate}
  \begin{sloppypar}
  \item The $2\mbox{-cochain}$ $\{c(g,h)\}_{g,h \in G}$ is independent of the isomorphisms $\{\phi_g\}_{g \in G}$ up to a coboundary.
  \item If $y' \in \objects(\gerbe_{U})$ is an object isomorphic to $y$, then there exist isomorphisms ${\{\phi'_{g}: y' \to F_g(y')\}}$ for which the resulting $2\mbox{-cochain}$ is identical to ${\{c(g,h)\}_{g,h \in G}}$. 
  \item The restriction of $\kappa_{(F,\alpha)}$ to $U$ is the
    coboundary of the $2\mbox{-cochain}$ $\{c(g,h)\}_{g,h \in G}$. 
  \end{sloppypar}
\end{enumerate}
\end{lemma}
\begin{proof}
  \begin{sloppypar}
    A different choice of isomorphisms is given by $\{a_g \cdot \phi_g\}_{g \in G}$, where $\{a_g\}_{g \in G}$ is a $1\mbox{-cochain}$ with values in $A(U)$. It is straightforward to see that this alters the cochain $\{c(g,h)\}_{g,h}$ by the coboundary of $\{a_g\}_{g \in G}$.

    Suppose $y' \xrightarrow{\lambda} y$ is an isomorphism in $\gerbe_U$. Then we can choose ${\phi'_g: y' \to \phi_g(y')}$ such following diagram commutes:
\[
  \begin{tikzcd}
    y \arrow{d}{\lambda} \arrow{r}{\phi_g} & F_g(y) \arrow{d}{F_g(\lambda)} \\
    y' \arrow{r}{\phi'_g} & F_g(y')
  \end{tikzcd}
\]
Then the resulting cochain for $y'$ is identical to $\{c(g,h)\}_{g,h \in G}$. 
  
  Let $g_1,g_2,g_3 \in G$ be three group elements. Since ${\alpha_{g_1,g_2}: F_{g_1g_2} \to F_{g_1} \circ F_{g_2}}$ is a natural transformation, we see that the following diagram commutes:
  \end{sloppypar}
\[
  \begin{tikzcd}[column sep = huge, row sep = large]
    F_{g_1g_2}(y) \arrow{d}{\alpha_{g_1,g_2}(y)} \arrow{r}{F_{g_1g_2}(\phi_{g_3})} & F_{g_1g_2} \circ F_{g_3}(y) \arrow{d}{\alpha_{g_1,g_2} (F_{g_3}(y))}\\
    F_{g_1} \circ F_{g_2} (y) \arrow{r}{F_{g_1} \circ F_{g_2}(\phi_{g_3})} &     F_{g_1} \circ F_{g_2}\circ F_{g_3} (y)    
  \end{tikzcd}
\]
\begin{sloppypar}
  We note that ${\alpha_{g_1,g_2} (F_{g_3}(y)) = (\alpha_{g_1,g_2} \star
  id_{F_{g_3}})(y)}$ so that with $\omega(g_1,g_2,g_3)$ as in Lemma \ref{lemma-commutative-diagram-for-local-objects-gerbe-F-alpha}, we have ${\omega(g_1,g_2,g_3) = a(g_1,g_2) \cdot \left( \alpha_{g_1,g_2} \star id_{F_{g_3}}\right)(y) }$. Also, we have ${F_{g_1}(\alpha_{g_2,g_3}(y)) = \left( id_{F_{g_1}} \star \alpha_{g_2,g_3}\right) (y)}$. By Lemma \ref{lemma-commutative-diagram-for-local-objects-gerbe-F-alpha} we have
$$F_{g_1}(\gamma(g_2,g_3)) \circ \, \gamma(g_1,g_2g_3)  = \omega(g_1,g_2,g_3) \circ \, \gamma(g_1g_2,g_3) $$ 
where we make the following substitutions:
\end{sloppypar}
\begin{align*}
  \gamma(g_1,g_2g_3) & = c(g_1,g_2g_3) \cdot \left( \alpha_{g_1,g_2}(y) \right) \\
  F_{g_1}(\gamma(g_2,g_3)) &= F_{g_1}(c(g_2,g_3) \cdot \alpha(g_2,g_3)(y)) \\
  \gamma(g_1g_2,g_3) & = c(g_1g_2,g_3) \cdot \left( \alpha_{g_1g_2,g_3}(y) \right) \\
  \omega(g_1,g_2,g_3) &= a(g_1,g_2) \cdot \left( \alpha_{g_1,g_2}
  \star id_{F_{g_3}}\right)(y) 
\end{align*}
Then, we see that (See Lemma \ref{lemma-G-action-on-2-morphisms-same-as-on-global-sections})
$$ \kappa_{F,\alpha}(g_1,g_2,g_3)|_{U} = c(g_1,g_2g_3) + \left(
  g_1 \cdot c(g_2,g_3) \right) - c(g_1g_2,g_3) -
c(g_1,g_2) .$$
This proves the last part of the lemma.
\end{proof}
\begin{proposition}\label{prop-3-cochain-is-cocycle-general-case}
  \begin{sloppypar}
    Suppose $(F,\alpha)$ is the data of a group action by $G$ on a gerbe $\gerbe$ banded by $A$. Then $\kappa_{(F,\alpha)}$ is a cocycle whose cohomology class in ${\cohomology{3}{G}{A(\xdown)}}$ is independent of the connecting $2\mbox{-morphisms}$ $\{\alpha_{g,h}\}_{g,h \in G}$.
  \end{sloppypar}
\end{proposition}
\begin{proof}
We have already seen that a different choice of connecting $2\mbox{-morphisms}$ alters $\kappa_{(F,\alpha)}$ by a coboundary (Lemma
\ref{lemma-different-connecting-changes-3-cocycle-by-a-coboundary}). There exists a cover ${\{U_i \to \xdown\}_{i \in I}}$ such that for each $U_i$, the conditions as in Situation \ref{situation-for-local-objects-for-gerbe-F-alpha} are met. By the above Lemma \ref{lemma-locally-kappa-f-alpha-is-coboundary}, we see that $\kappa_{(F,\alpha)}$ is a coboundary on each $U_i$. Therefore, $\kappa_{(F,\alpha)}$ is a cocycle.
\end{proof}
\begin{definition}\label{definition-obstruction-class}
Let $y$ be an $A\mbox{-torsor}$ on $\xdown$ and let $\gerbe$ be a gerbe on $\xdown$ which is banded by $A$. Given an action of the group $G$ on $y$ we refer to the associated cohomology class in ${\cohomology{2}{G}{A(\xdown)}}$ as the obstruction cohomology class. Similarly, given an action of $G$ on $\gerbe$ we refer the associated cohomology class in ${\cohomology{3}{G}{A(\xdown)}}$ as the obstruction cohomology class. 
\end{definition}
\begin{sloppypar}
  Previously, given an $A\mbox{-torsor}$ together with a $G\mbox{-action}$, we defined a pseudo torsor over the invariants. Following is the analogous construction for gerbes. Let $\gerbe$ be a gerbe banded by $A$. Suppose $(F,\alpha)$ is the data of an action of $G$ on $\gerbe$ together with connecting $2\mbox{-morphisms}$. We define a category $\gerbe_{(F,\alpha)}$ as follows. An object of $\gerbe_{(F,\alpha)}$ is a tuple ${(U,y,\{\phi_g\}_{g \in G})}$ where 
\begin{itemize}
\item $U \in \objects(\thesite)$ is an object of $\thesite$.
\item $y \in \objects(\gerbe_{U})$ is an object of $\gerbe$ over $U$. 
\item For $g \in G$, ${\phi_{g}: y \to F_g(y)}$ is a morphism in $\gerbe_{U}$ such that for $g,h \in G$, the following diagram commutes:
\[
  \begin{tikzcd}
    y \arrow{d}[swap]{\phi_{g}} \arrow{r}{\phi_{g \cdot h}} & F_{g \cdot h}(y) \arrow{d}{\alpha_{g,h}} \\
    F_{g}(y) \arrow{r}{F_g(\phi_{h})} & \left(F_{g} \circ F_{h}\right)(y) 
  \end{tikzcd}
\]
\end{itemize}
A morphism ${(U_1,y_1,\{\phi^1_g\}_{g \in G}) \to (U_2,y_2,\{\phi^2_g\}_{g \in G})}$ in $\gerbe_{(F,\alpha)}$ is a pair $(f,\lambda)$, where ${f: U_1 \to U_2}$ is a morphism in $\thesite$, and ${\lambda: y_1 \to y_2}$ is a morphism in $\gerbe$ over $f$ such that the following diagram commutes:
\[
  \begin{tikzcd}
    y_1 \arrow{d}[swap]{\phi^1_g} \arrow{r}{\lambda} & y_2 \arrow{d}{\phi^2_g}  \\
    F_g(y_1) \arrow{r}{F_g(\lambda)} & F_g(y_2)
  \end{tikzcd}
\]
We can pretend that the functors $\{F_g\}_{g \in G}$ commute with the pull back functors (See Appendix I, Section \ref{section-fibered-cat-pullbacks}). Therefore, it is easy to see that the objects of $\gerbe_{(F,\alpha)}$ can be pulled back, making it a fibered category in groupoids over $\thesite$. Furthermore, we see that $\gerbe_{(F,\alpha)}$ is a stack in groupoids. We have a morphism ${\gerbe_{(F,\alpha)} \to \gerbe}$ over $\thesite$ which is given by ${(U,y,\{\phi_g\}_{g \in G}) \to y}$.
\end{sloppypar}
\begin{lemma}\label{lemma-objects-isomorphic-in-gerbe_F_alpha}
  \begin{sloppypar}
    Suppose ${Y_1 = \left(U,y_1,\{\phi^1_{g}\}_{g \in G} \right)}$ and ${Y_2 = \left(U,y_2,\{\phi^2_{g}\}_{g \in G} \right)}$ are objects of $\gerbe_{(F,\alpha)}$ such that there exists an isomorphism ${\lambda: y_1 \to y_2}$ in $\gerbe_{U}$. For $g \in G$, we let $a_g \in A(U)$ denote the scalar such that the following diagram commutes:
  \end{sloppypar}
  \[
  \begin{tikzcd}
    y_1 \arrow{d}[swap]{\lambda} \arrow{r}{\phi^1_g} & F_g(y_1) \arrow{d}{a_g \cdot F_g(\lambda)} \\
    y_2 \arrow{r}{\phi^2_g} & F_g(y_2) 
  \end{tikzcd}
\]
Then $\{a_g\}_{g \in G}$ is a $1\mbox{-cocycles}$ with values in $A(U)$. Moreover, its cohomology class in $\cohomology{1}{G}{A(U)}$ is zero if and only if $Y_1$ and $Y_2$ are isomorphic objects of $\gerbe_{(F,\alpha)}$ over $U$. 
\end{lemma}
\begin{proof}
Let $g,h \in G$ be two elements of the group. We know that $a_h$ is so that the following diagram commutes:
\[
  \begin{tikzcd}
    y_1 \arrow{d}[swap]{\lambda} \arrow{r}{\phi^1_h} & F_h(y_1) \arrow{d}{a_h \cdot F_h(\lambda)} \\
    y_2 \arrow{r}{\phi^2_h} & F_h(y_2) 
  \end{tikzcd}
\]
By applying $F_g$, we conclude that the following diagram commutes as well:
\[
  \begin{tikzcd}[column sep = huge]
    F_g(y_1) \arrow{d}[swap]{F_g(\lambda)} \arrow{r}{F_g(\phi^1_h)} & (F_g \circ F_h)(y_1) \arrow{d}{(g \cdot a_h) \cdot (F_g \circ F_h)(\lambda)} \\
    F_g(y_2) \arrow{r}{F_g(\phi^2_h)} & (F_g \circ F_h)(y_2) 
  \end{tikzcd}
\]
Now we can verify that the following diagram commutes: 
\[
  \begin{tikzcd}[column sep = huge, row sep = huge]
  y_1 \arrow{d}{\lambda} \arrow{r}{\phi^1_{g}} & F_{g}(y_1) \arrow{d}{a_{g} \cdot F_{g}(\lambda)}\arrow{r}{F_g(\phi^1_h)} & (F_g \circ F_h)(y_1)  \arrow{d}{(a_g + g \cdot a_{h}) \cdot (F_g \circ F_h)(\lambda)} \\
  y_2 \arrow{r}{\phi^2_{g}} & F_{g}(y_2) \arrow{r}{F_g(\phi^2_h)} & (F_g \circ F_h)(y_2) 
  \end{tikzcd}
\]
Since ${\alpha_{g,h}: F_{gh} \to F_{g} \circ F_{h}}$ is a base preserving natural transformation, we see that the following diagram commutes:
\[
  \begin{tikzcd}[column sep = huge, row sep = huge]
  y_1 \arrow{d}{\lambda} \arrow{r}{\phi^1_{g \circ h}} & F_{g \circ h}(y_1) \arrow{d}{a_{g \circ h} \cdot F_{g \circ h}(\lambda)}\arrow{r}{\alpha_{g,h}} & (F_g \circ F_h)(y_1)  \arrow{d}{a_{gh} \cdot (F_g \circ F_h)(\lambda)} \\
  y_2 \arrow{r}{\phi^2_{g \circ h}} & F_{g \circ h}(y_2) \arrow{r}{\alpha_{g,h}} & (F_g \circ F_h)(y_2)   
  \end{tikzcd}
\]
\begin{sloppypar}
  For $j = 1 \text{ and } 2$, the two composite morphisms ${y_j \to (F_g \circ F_h)(y_j)}$, namely ${F_g(\phi^j_{h}) \circ \phi^j_{g}}$ and ${\alpha_{g,h}(y_j) \circ \phi^j_{gh}}$, which appear in the above commutative diagrams are equal. Therefore, we conclude that ${a_{gh} = a_g + (g \cdot a_h)}$. Thus, we see that $\{a_g\}_{g \in G}$ is a $1\mbox{-cocycle}$.

If $\{a_g\}_{g \in G}$ is a coboundary, then there exists ${ a \in A(U)} $ such that for $g \in G$, we have ${a_g = (g \cdot a) - a}$. Then we see that 
${a \cdot \lambda: y_1 \to y_2}$ gives an isomorphism of $Y_1$ and $Y_2$ in the fiber category of $\gerbe_{(F,\alpha)}$ over $U$.
\end{sloppypar}
\end{proof}
\begin{lemma}\label{lemma-gerbe-f-alpha-is-psuedo-gerb-when-first-group-cohomology-vanishes-locally}
If the first group cohomology classes locally vanish for the underlying action on $A$, then the category $\gerbe_{(F,\alpha)}$ is a pseudo abelian gerbe over $\xdown$. 
\end{lemma}
\begin{proof}
  \begin{sloppypar}
    We have already observed that $\gerbe_{(F,\alpha)}$ is a stack in groupoids over $\xdown_{\tau}$. Suppose we have an object ${Y_1 = \left(U,y_1,\{\phi^1_{g}\}_{g \in G} \right)}$ of $\gerbe_{(F,\alpha)}$ over an object $U$ of $\thesite$. Then the automorphisms of $Y_1$ over $U$ are given by isomorphisms $\lambda: y_1 \to y_1$ in $\gerbe_{U}$ such that for $g \in G$, the following diagram commutes: \end{sloppypar}
  \[
  \begin{tikzcd}
    y_1 \arrow{d}[swap]{\phi^1_g} \arrow{r}{\lambda} & y_1 \arrow{d}{\phi^1_g}  \\
    F_g(y_1) \arrow{r}{F_g(\lambda)} & F_g(y_1)
  \end{tikzcd}
\]
The morphism $F_g$ is equivariant with respect to the action of $g$ on the band $A$. Thus, we see that the isom sheaf $\isom(Y_1,Y_1)$ on $\thesite/U$ is isomorphic to ${\invariants}|_U$ and in particular, it is abelian. Suppose ${Y_2 = \left(U,y_2,\{\phi^2_{g}\}_{g \in G} \right)}$ is another object of $\gerbe_{(F,\alpha)}$ over $U$. Since $\gerbe$ is a gerbe, there exists a cover of $U$ where the objects $y_1$ and $y_2$
are isomorphic. If the first group cohomology classes locally vanish, then it follows from Lemma \ref{lemma-objects-isomorphic-in-gerbe_F_alpha} that there exists a further refinement of this cover where $Y_1$ and $Y_2$ are isomorphic in $\gerbe_{(F,\alpha)}$. 
\end{proof}
\begin{lemma}\label{lemma-2-cocycle-obstruction-for-local-objects-for-gerbe-F-alpha}
  \begin{sloppypar}
    Let $U \in \objects(\thesite)$ be an object over which $\gerbe$ admits a section $y$ in $\gerbe_{U}$ such that the objects $\{F_g(y)\}_{g \in G}$ are all isomorphic to $y$ in $\gerbe_U$. Suppose the restriction of the $3\mbox{-cochain}$ $\kappa_{(F,\alpha)}$ to $U$ is identically zero. Then, there is a cohomology class in ${\cohomology{2}{G}{A(U)}}$ which is zero if and only if there exists a choice of isomorphisms ${\{\psi_g: y \to F_g(y)\}_{g \in G}}$ such that ${(U,y,\{\psi_g\}_{g \in G})}$ is a valid object of $\gerbe_{(F,\alpha)}$ over $U$.
  \end{sloppypar}
\end{lemma}
\begin{proof}
  \begin{sloppypar}
    We choose isomorphisms ${\{\phi_g: y \to F_g(y)\}_{g \in G}}$ and define scalars $\{c(g,h)\}_{g,h}$ in $A(U)$ as in Situation \ref{situation-for-local-objects-for-gerbe-F-alpha}. If the $3\mbox{-cochain}$ $\kappa_{(F,\alpha)}$ vanishes on $U$, then by Lemma \ref{lemma-locally-kappa-f-alpha-is-coboundary} we see that $\{c(g,h)\}_{g,h \in G}$ is a cocycle. We also conclude from Lemma \ref{lemma-locally-kappa-f-alpha-is-coboundary} that its cohomology class is zero if and only if there exist scalars $\{a_g\}_{g \in G}$ in $A(U)$ such that ${(U,y,\{a_g \cdot \phi_g \})}$ is a legitimate object of $\gerbe_{(F,\alpha)}$ over $U$.
  \end{sloppypar}
\end{proof}
\begin{lemma}\label{lemma-when-gerbe-f-alpha-is-legit-gerbe}
Assume that the first two group cohomology classes locally vanish for the underlying action of $G$ on $A$. Then $\gerbe_{(F,\alpha)}$ is an abelian gerbe over $\xdown$ if and only if the associated $3\mbox{-cocycle}$ $\kappa_{(F,\alpha)}$ is identically zero.
\end{lemma}
\begin{proof}
Suppose $U \in \objects(\thesite)$ is such that the fiber of $\gerbe_{(F,\alpha)}$ over $U$ is non-empty. Then we see from Lemma \ref{lemma-locally-kappa-f-alpha-is-coboundary} that $\kappa_{(F,\alpha)}$ vanishes on $U$. Therefore, if $\gerbe_{(F,\alpha)}$ admits local sections, we conclude that $\kappa_{(F,\alpha)}$ vanishes everywhere. 

Conversely, suppose $\kappa_{(F,\alpha)}$ is identically zero. If the first cohomology classes locally vanish, then by Lemma \ref{lemma-gerbe-f-alpha-is-psuedo-gerb-when-first-group-cohomology-vanishes-locally} we know that $\gerbe_{(F,\alpha)}$ is a pseudo abelian gerbe. Since $\gerbe$ is a gerbe, there exists a cover ${\{U_i \to \xdown\}}$ of $\xdown$ together with sections $y_i$ in $\gerbe_{U_i}$ such that the conditions as in Situation \ref{situation-for-local-objects-for-gerbe-F-alpha} are met for each $U_i$. Then by
Lemma \ref{lemma-2-cocycle-obstruction-for-local-objects-for-gerbe-F-alpha}, there is a cohomology class in ${\cohomology{2}{G}{A(U_i)}}$  which is zero if and only if $y_i$ gives an object of $\gerbe_{(F,\alpha)}$ over $U_i$. If the second group cohomology classes locally vanish, then there exists a refinement of this cover over which the pullbacks of the sections $y_i$ produce sections of $\gerbe_{(F,\alpha)}$. 
\end{proof}
\begin{sloppypar}
  The constructions of the torsor $y_{\lambda}$ and the gerbe $\gerbe_{(F,\alpha)}$ are functorial in the following sense. Let ${f: A_1 \to A_2}$ me a morphism of abelian sheaves on $\thesite$.
  \begin{remark}\label{remark-functoriality-of-y-alpha-gerbe-f-constructions} Suppose we have a $G\mbox{-equivariant}$ map of torsors $y_1 \to y_2$ as in Definition \ref{def-G-equivariant-morphism-of-torsors}. We have the following commutative diagram of abelian sheaves:
    \[
      \begin{tikzcd}
        A_1^{G} \arrow{d}{i_1} \arrow{r}{f^G} & A_2^{G} \arrow{d}{i_2} \\
        A_1 \arrow{r}{f} & A_2
      \end{tikzcd}
    \]
    Then there is a map ${{y_1}_{\lambda^1} \to {y_2}_{\lambda^2}}$ of pseudo-torsors which is  $f^{G}\mbox{-equivariant}$.
  \end{remark}
\noindent The analogous result for gerbes is the following:
\end{sloppypar}
\begin{lemma}\label{lemma-functoriality-of-gerbe-f-alpha}
  \begin{sloppypar}
    Let ${F: \gerbe_1 \to \gerbe_2}$ be an $f\mbox{-equivariant}$ morphism of gerbes which is ${G\mbox{-equivariant}}$ as in Definition \ref{def-G-equivariant-morphism-of-gerbes}. Let $\{\alpha^1_{g,h}\}_{g,h \in G}$ be a choice of connecting $2\mbox{-morphisms}$ for the action on $\gerbe_1$. Let ${\gamma_g: F \circ F^1_g \to F^2_g \circ F}$ be a choice of a $2\mbox{-morphism}$. Let $\{\alpha^2_{g,h}\}_{g,h \in G}$ be the uniquely determined connecting $2\mbox{-morphisms}$ for $\gerbe_2$ as in Proposition \ref{prop-g-equivariant-map-and-3-cochain}. Then there is a morphism
    ${{\gerbe_1}_{(F,\alpha^1)} \to {\gerbe_2}_{(F,\alpha^2)}}$
    of fibered categories over $\xdown_{\tau}$ which fits in the $2\mbox{-commutative}$ diagram below:
  \end{sloppypar}
  \[
  \begin{tikzcd}
    {\gerbe_1}_{(F,\alpha^1)} \arrow{r}\arrow{d}{} & {\gerbe_2}_{(F,\alpha^2)} \arrow{d}{} \\
    \gerbe_1 \arrow{r}{F} & \gerbe_2
  \end{tikzcd}
\]
\end{lemma}
\begin{proof}
\begin{sloppypar}
Suppose $(U,y,\{\phi_g\}_{g \in G})$ is an object of ${\gerbe_1}_{(F,\alpha^1)}$ over an object $U$ in $\thesite$. Then for $g \in G$, we have the morphism
\[
  \begin{tikzcd}
    F(y) \arrow{r}{F(\phi_g)} & (F \circ F^{1}_{g})(y) \arrow{r}{\gamma_g(y)} & (F^{2}_{g} \circ F)(y).
  \end{tikzcd}
\]
It is straightforward to verify that ${(U,F(y),\{\gamma_g(y) \circ F(\phi_g)\}_{g \in G})}$ is a valid object of ${\gerbe_2}_{(F,\alpha^2)}$ over $U$. Thus, we have a morphism ${{\gerbe_1}_{(F,\alpha^1)} \to {\gerbe_2}_{(F,\alpha^2)}}$ defined by
\end{sloppypar}
$$ (U,y,\{\phi_g\}_{g \in G})  \to (U,F(y),\{\gamma_g(y) \circ F(\phi_g)\}_{g \in G}). $$
\end{proof}
\begin{corollary}\label{corollary-change-connecting-morphisms-gives-same-gerbe-f-alpha}
  \begin{sloppypar}
    Suppose we have two actions of $G$ on a gerbe $\gerbe$ banded by $A$ that are given by morphisms $\{F_g\}_{g \in G}$ and $\{F'_g\}_{g \in G}$. Suppose for $g \in G$, we have a $2\mbox{-isomorphism}$ ${\gamma_g: F_g \to F'_g}$. Let $\{\alpha_{g,h}\}_{g,h \in G}$ be a choice of connecting $2\mbox{-morphisms}$ for the first action. We fix the connecting $2\mbox{-morphisms}$ for the second action such that the following diagram commutes:
  \end{sloppypar}
   \[ 
\begin{tikzcd}
       F'_{gh} \arrow{d}{\gamma_{gh}} \arrow{r}{\alpha'_{g,h}} & F'_g \circ F'_h \arrow{d}{\gamma_g \star \gamma_h} \\
       F_{gh} \arrow{r}{\alpha_{g,h}} & F_g \circ F_h 
     \end{tikzcd}
 \]
Then the categories $\gerbe_{(F,\alpha)}$ and $\gerbe_{(F',\alpha')}$ are isomorphic over $\thesite$. 
\end{corollary}
\begin{proof}
  \begin{sloppypar}
    In this situation, the identity map of $\gerbe$ is a $G\mbox{-equivariant}$ morphism with respect to the two actions. With the choice of $2\mbox{-isomorphisms}$ $\{\gamma_g\}_{g \in G}$, we apply the above Lemma \ref{lemma-functoriality-of-gerbe-f-alpha} to see that we have a morphism ${\gerbe_{(F,\alpha)} \to \gerbe_{(F,\alpha')}}$ of categories over $\thesite.$ The inverse morphism is given by considering the inverses of the $2\mbox{-morphisms}$ $\{\gamma_g\}_{g \in G}$ and then applying the same lemma.
  \end{sloppypar}
\end{proof}
We fix an action of the group $G$ on the abelian sheaf $A$. Let $\invariants \xrightarrow{i} A$ denote the sheaf of invariants.
\begin{definition}\label{def-G-action-on-torsors-and-gerbes}
\begin{sloppypar}
If $y$ is an $A\mbox{-torsor}$, we denote by $T_g(y)$ the $A\mbox{-torsor}$ which as a sheaf of sets is the same as $y$, but the action of the scalars is  twisted by $g^{-1}$. That is, the torsor structure on $T_g(y)$ is given by the composite map below:
  \[
    \begin{tikzcd}[column sep = large]
      A \times y \arrow{r}{g^{-1} \times id_y} & A \times y \arrow{r}{} & y
    \end{tikzcd}
  \] 
Similarly, if $\gerbe$ is an abelian gerbe banded by $A$, then we denote by $T_g(\gerbe)$ (ignoring a slight abuse of notation) the gerbe, which as a fibred category over $\thesite$ is the same as $\gerbe$ but where we twist the isomorphism ${\text{Band}(\gerbe) \to A}$ by $g^{-1}$. It is easy to see that an $A\mbox{-equivariant}$ morphism to $T_g(y)$ (resp. to $T_g(\gerbe)$) is the same as a morphism to $y$ (resp. to $\gerbe$) which is equivariant with respect to the map ${A \xrightarrow{a \to (g^{-1} \cdot a) } A}$.
\begin{remark}\label{remark-tg-operators-define-action-on-sheaf-cohomology}
  \begin{sloppypar}
The operators $\{T_g\}_{g \in G}$ describe the induced action of $G$ on the sheaf cohomology groups $\cohomology{1}{\xdown}{A}$ and  $\cohomology{2}{\xdown}{A}$ in terms of torsors and gerbes. It is easy to see that an $A\mbox{-torsor}$ or a gerbe banded banded by $A$ represents a $G\mbox{-stable}$ sheaf cohomology class if and only if it admits a lift of the $G\mbox{-action}$. 
  \end{sloppypar}
\end{remark}
  \end{sloppypar}
\end{definition}
\begin{proposition}\label{prop-induced-torsor-and-gerbes}
  \begin{sloppypar}
    Let $y$ be an $A\mbox{-torsor}$ on $\xdown$ which is induced from the invariants. Then $y$ admits a lift of the underlying action on $A$ such that the associated obstruction class in ${\cohomology{2}{G}{A(\xdown)}}$ is zero. Similarly, if $\gerbe$ is a gerbe banded by $A$ which is induced from the invariants then $\gerbe$ admits a lift of the action for which the associated obstruction class in ${\cohomology{3}{G}{A(\xdown)}}$ is zero. More concretely, we have the following:
    \begin{itemize}
      \begin{sloppypar}
      \item Let $\overline{y}$ be an $\invariants\mbox{-torsor}$ with a morphism ${f:\overline{y} \to y }$ that is $i\mbox{-equivariant}$. Then there are unique automorphisms $\{\lambda_g\}_{g \in G}$ of $y$ that lift the $G\mbox{-action}$ such that ${\lambda_g \circ f = f}$ for all $g \in G$. The associated cochain $\chi_{\lambda}$ for this action on $y$ is zero.
      \item Let $\overline{\gerbe}$ be an abelian gerbe  banded by $\invariants$ with an $i\mbox{-equivariant}$ morphism ${F: \overline{\gerbe} \to \gerbe }$. Then there are unique (up-to $2\mbox{-isomorphisms}$) automorphisms ${\{F_g\}_{g \in G}}$ of $\gerbe$ that lift the $G\mbox{-action}$ such that the morphisms $F_g \circ F$ and $F$ are $2\mbox{-isomorphic}$ for all $g \in G$. The associated obstruction cohomology class $\kappa_F$ for this action on $\gerbe$ is zero.
      \end{sloppypar}
    \end{itemize}
  \end{sloppypar}
\end{proposition}
\begin{proof}
  \begin{sloppypar}
    For $g \in G$, we have a map of sheaves ${f: \overline{y} \to T_{g^{-1}}(y) }$ which is also an $i\mbox{-equivariant}$ map of torsors. By Proposition \ref{prop-hom-bands-gerbe-construction}, there is a unique morphism ${\lambda_g: y \to y}$ which is equivariant with respect the action of $g$ on $A$ such that ${\lambda_g \circ f = f}$. For $g,h \in G$, we have two morphisms, $\lambda_{gh}$ and ${\lambda_{g} \circ \lambda_{h}}$, which are equivariant with respect to the action by $gh$. Moreover, we have ${\lambda_{gh} \circ f = f}$ and ${\lambda_{g} \circ \lambda_{h} \circ f = f}$. By uniqueness, we conclude that ${\lambda_g \circ \lambda_h = \lambda_{gh}}$. Thus, we see that the associated $2\mbox{-cocycle}$ $\chi_{\lambda}$ is zero. 

    For $g \in G$, $F$ also gives a morphism ${\overline{\gerbe} \to T_{g^{-1}}(\gerbe) }$ which is $i\mbox{-equivariant}$. By Proposition \ref{prop-hom-bands-gerbe-construction}, there exists a morphism ${F_g: \gerbe \to \gerbe }$ which is unique up to a $2\mbox{-isomorphism}$ such that ${F_g \circ F}$ and $F$ are $2\mbox{-isomorphic}$. By exploiting the uniqueness as above, it follows that for $g,h \in G$, the morphisms $F_{gh}$ and ${F_g \circ F_h}$ are $2\mbox{-isomorphic}$. Thus, we have the following $2\mbox{-commutative}$ diagram:
  \end{sloppypar}
  \[
  \begin{tikzcd}
    \overline{\gerbe} \arrow{d}{id} \arrow{r}{F} & \gerbe \arrow{d}{F_g} \\
    \overline{\gerbe} \arrow{r}{F} & \gerbe
  \end{tikzcd}
\]
We see that $F$ is a $G\mbox{-equivariant}$ morphism of gerbes with respect to the trivial action by $G$ on $\overline{\gerbe}$. It follows from Proposition \ref{prop-g-equivariant-map-and-3-cochain}, that the associated obstruction class for the action on $\gerbe$ is trivial. We note that we can also use a similar argument for torsors (See Lemma \ref{lemma-functoriality-of-2-cochain-for-torsors}). 
\end{proof}
\begin{theorem}\label{theorem-induced-from-invariants-when-cohomology-class-is-zero}
We have the converse of the above Proposition \ref{prop-induced-torsor-and-gerbes} when the group cohomology classes locally vanish. (See Appendix II (Section \ref{section-necessity-of-local-vanishing}) for examples that show that the local vanishing is necessary.)
  \begin{itemize}
    \begin{sloppypar}
    \item Let $y$ be an $A\mbox{-torsor}$ that admits a lift of the underlying action on $A$ such that the associated obstruction class in ${\cohomology{2}{G}{A(\xdown)}}$ is zero. If the first group cohomology classes locally vanish for $A$, then $y$ is induced from the invariants. 
    \item Let $\gerbe$ be an abelian gerbe banded by $A$ which admits a lift of the underlying action on $A$ such that the associated obstruction class in ${\cohomology{3}{G}{A(X)}}$ is zero. If the first two group cohomology classes locally vanish for $A$, then $\gerbe$ is induced from the invariants.
    \end{sloppypar}
  \end{itemize}
More concretely, we prove the following:
\begin{enumerate}
\item
  \begin{sloppypar}
    Assume the first group cohomology classes locally vanish for $A$. Suppose we have automorphisms $\{\lambda_g\}_{g \in G}$ of $y$ that lift the action of $G$ such that $\chi_{\lambda}$ is identically zero. Then the following holds:
  \end{sloppypar}
  \begin{itemize}
    \begin{sloppypar}
    \item There exists an ${\invariants\mbox{-torsor}}$ $\overline{y}$ with an $i\mbox{-invariant}$ morphism $f: \overline{y} \to y$ such that $\lambda_g \circ f = f $ for all $g \in G$.
    \item An ${\invariants\mbox{-torsor}}$ $\overline{y}$ as above is unique up to an ${\invariants\mbox{-equivariant}}$ isomorphism.
    \end{sloppypar}
  \end{itemize}
   
\item
  \begin{sloppypar}
    Assume that the first two group cohomology classes locally vanish for $A$. Suppose we have automorphisms $\{F_g\}_{g \in G}$ of $\gerbe$ that lift the $G\mbox{-action}$ such that the cohomology class $\kappa_{F}$ is zero. Then the following holds:
  \end{sloppypar}
  \begin{itemize}
    \begin{sloppypar}
    \item There exists a gerbe $\overline{\gerbe}$ banded by $\invariants$ with an $i\mbox{-equivariant}$ morphism ${F: \overline{\gerbe} \to \gerbe}$ such that for $g \in G$, the morphisms $F$ and $F_g \circ F$ are $2\mbox{-isomorphic}$. 
    \item For any gerbe $\overline{\gerbe}$ banded by $\invariants$ as above, we have an ${\invariants\mbox{-equivariant}}$ isomorphism ${\overline{\gerbe} \cong \gerbe_{(F,\alpha')}}$ for some choice of connecting $2\mbox{-isomorphism}$ ${\{\alpha'_{g,h}\}_{g,h \in G}}$.
    \end{sloppypar}
  \end{itemize}
\end{enumerate}
\end{theorem}
\begin{proof}
  \begin{sloppypar}
    If $\chi_{\lambda}$ is identically zero, then by Lemma \ref{lemma-y-lambda-admits-local-sections-iff-general-case} we see that $y_{\lambda}$ is an $\invariants\mbox{-torsor}$. We have a natural map ${f :y_{\lambda} \to y}$ which is $i\mbox{-equivariant}$. Moreover, it is clear that for $g \in G$, we have ${\lambda_{g} \circ f = f}$. Suppose ${f': \overline{y} \to y}$ is another $i\mbox{-equivariant}$ morphism of torsors such that ${\lambda_{g} \circ f' = f'}$ for all $g \in G$. For an object $U$ of $\thesite$, $f'$ gives a map ${f': \overline{y}(U) \to y(U)}$. Since ${\lambda_{g} \circ f' = f'}$ for all $g \in G$, we see that for any $x \in \overline{y}(U)$, we have ${\lambda_{g}(f'(x)) = f'(x)}$. Thus the map ${x \to f'(x)}$ gives an $\invariants\mbox{-equivariant}$ morphism from $\overline{y}$ to $y_{\lambda}$. Since equivariant morphisms of torsors are isomorphisms, this proves the first part.

    By Lemma \ref{cor-3-cochain-depends-only-on-2-iso-class-upto-coboundary}, we can choose connecting ${2\mbox{-isomorphisms}}$ ${\{\alpha_{g,h}\}_{g,h \in G}}$ such that the cochain $\kappa_{(F,\alpha)}$ is identically zero. Then by Lemma \ref{lemma-when-gerbe-f-alpha-is-legit-gerbe}, we see that $\gerbe_{(F,\alpha)}$ is an abelian gerbe banded by $\invariants$. There is a natural $i\mbox{-equivariant}$ morphism of gerbes ${F: \gerbe_{(F,\alpha)} \to \gerbe}$ which is locally given by ${(U,y,\{\phi_g\}_{g \in G}) \to y}$. For $g \in G$, we have a base preserving natural transformation locally given by ${y \xrightarrow{\phi_g} F_g(y)}$ that defines a $2\mbox{-isomorphism}$ from $F$ to $F_g \circ F$. Suppose ${F': \overline{\gerbe} \to \gerbe}$ is another $i\mbox{-equivariant}$ morphism of gerbes such that ${F_{g} \circ F'}$ and $F'$ are $2\mbox{-isomorphic}$ for all $g \in G$. We consider trivial action of $G$ on $\overline{\gerbe}$ with trivial connecting $2\mbox{-isomorphisms}$. Then we have a situation as in Proposition \ref{prop-g-equivariant-map-and-3-cochain}, where we have a $G\mbox{-equivariant}$ map ${F': \overline{\gerbe} \to \gerbe}$. We choose $2\mbox{-morphisms}$ ${\{\epsilon_g: F \to F_g \circ F\}_{g \in G}}$ and fix the connecting $2\mbox{-isomorphisms}$ $\{\alpha'_{g,h}\}_{g,h \in G}$ for $\gerbe$ as in Proposition \ref{prop-g-equivariant-map-and-3-cochain}. Consequently, the cochain $\kappa_{(F,\alpha')}$ is identically zero. Since the first two group cohomology classes locally vanish, we can say that $\gerbe_{(F,\alpha')}$ is a gerbe banded by $\invariants$. Suppose $U$ is an object of $\thesite$ and ${y \in \objects({\overline{\gerbe}_U})}$ is a section of $\overline{\gerbe}$ over $U$. Then for $g \in G$, we have the morphism ${\epsilon_g(y): F(y) \to \left(F_g \circ F\right) (y)}$ in the fiber category $\gerbe_U$. It is straightforward to verify that ${(U,F(y),\{\epsilon_g(y)\}_{g \in G})}$ is a valid object of $\gerbe_{(F,\alpha')}$ over $U$. Thus, ${y \to (U,F(y),\{\epsilon_g(y)\}_{g \in G})}$ defines a $\invariants\mbox{-equivariant}$ morphism of gerbes. Since equivariant morphisms of gerbes are isomorphisms, this completes the proof. 
  \end{sloppypar}
\end{proof}
\section{Group actions on trivial gerbe}
Let $G$ be a finite group and $A$ be an abelian sheaf on $\thesite$ as before. In this section, we consider group actions on the trivial gerbe banded by $A$. Our first goal is to prove that if $G$ acts equivariantly on the trivial gerbe (i.e. the underlying action on $A$ is trivial), then the obstruction class is zero.
\begin{lemma}\label{lemma-inflation-for-torsors}
  \begin{sloppypar}
    Let $y$ be an $A\mbox{-torsor}$. Suppose we are given an action of $G$ on $y$. Let $N$ denote the normal subgroup of $G$ of those elements whose underlying action on $A$ is trivial. Then we have an action of $G/N$ on $y$ such that the obstruction class for the $G\mbox{-action}$ is given by inflating the obstruction class for $G/N$ via the inflation map ${\cohomology{2}{G/N}{A(\xdown)} \to  \cohomology{2}{G}{A(\xdown)}}$.
\end{sloppypar}
\end{lemma}
\begin{proof}
Suppose the action of $G$ on $y$ is given by the morphisms $\{\lambda_g\}_{g \in G}$. Since $N$ acts trivially on $A$ and the obstruction class for $G$ does not depend on the choice of the lift (Lemma \ref{lemma-change-in-lifted-action-on-torsors-coboundary}), we can assume that whenever $g$ and $g'$ are congruent modulo $N$, we have ${\lambda_g = \lambda_{g'}}$. Thus we get an action of $G/N$ on $y$. Moreover, it is clear that the cochain that defines the obstruction class factors through $G/N$. Therefore it follows that the obstruction class is given by the inflation map as claimed.
\end{proof}
\begin{lemma}\label{lemma-inflation-for-gerbes}
  \begin{sloppypar}
    Let $\gerbe$ be a gerbe banded by $A$ over $\xdown$. Suppose we have an action of $G$ on $\gerbe$. Let $N$ denote the normal subgroup of $G$ of those elements whose action on $\gerbe$ is $2\mbox{-isomorphic}$ to the identity morphism. Then we have an action of $G/N$ on $\gerbe$ such that the obstruction class for $G$ is given by inflating the obstruction class for $G/N$ via the inflation map ${\cohomology{3}{G/N}{A(\xdown)} \to  \cohomology{3}{G}{A(\xdown)}}$.
  \end{sloppypar}
\end{lemma}
\begin{proof}
  \begin{sloppypar}
    Let $\{F_g\}_{g \in G}$ denote the automorphisms of $\gerbe$ through which $G$ acts on $\gerbe$. Since the obstruction class only depends on the $2\mbox{-isomorphism}$ classes of $\{F_g\}$ (Corollary \ref{cor-3-cochain-depends-only-on-2-iso-class-upto-coboundary}), we can assume that ${F_g = F_{g'}}$ whenever $g$ and $g'$ are congruent modulo $N$. Then it follows that we get an action of $G/N$ and the cochain that defines the obstruction class factors through $G/N$. Therefore we conclude that the obstruction class is given by the inflation map as claimed.
  \end{sloppypar}
\end{proof}
\begin{proposition}\label{proposition-cyclic-group-trivial-action-on-A}
Suppose $G$ is a cyclic group that acts on a gerbe banded by $A$ such that the underlying action on $A$ is trivial. Then the associated obstruction cohomology class in ${\cohomology{3}{G}{A(\xdown)}}$ is trivial. 
\end{proposition}
\begin{proof}
  \begin{sloppypar}
 Let $n$ be the order of the group $G$ and let $j > 0$ be an odd index. Since $G$ is a cyclic group acting trivially, for $U$ in $\objects(\thesite)$, we have a functorial isomorphism ${\cohomology{j}{G}{A(U)} \cong \{ a \in A(U) | n \cdot a = 0 \}}$. Consequently, the assignment ${U \to \cohomology{3}{G}{A(U)}}$ defines a sheaf on $\thesite$ which is isomorphic to the sheaf of $n\mbox{-torsion}$ elements in $A$. We know that the restriction of the obstruction class is locally zero (See proof of Proposition \ref{prop-3-cochain-is-cocycle-general-case}). By the sheaf condition, we conclude that it is globally zero.
  \end{sloppypar}
\end{proof}
\begin{lemma}\label{lemma-direct-product-action-on-gerbes}
  \begin{sloppypar}
    Let $\gerbe_1$ and $\gerbe_2$ be two gerbes banded by $A$. Suppose we have a group action by $G_1$ (resp. by $G_2$) on $\gerbe_1$ (resp. on $\gerbe_2$) such that the underlying action on $A$ is trivial.  Let $\gerbe_1 \times \gerbe_2$ denote the product category which is a gerbe banded by $A \oplus A$.  We consider component wise group action of $G_1 \times G_2$ on $\gerbe_1 \times \gerbe_2$. Let $\gerbe_3$ be the gerbe banded by $A$ which admits a morphism ${F: \gerbe_1 \times \gerbe_2 \to \gerbe_3}$ that is equivariant with respect to the group law. Then for every $(g_1,g_2)$ in $G_1 \times G_2$, there is a unique (up to $2\mbox{-isomorphism}$) $A\mbox{-equivariant}$ automorphism $F_{(g_1,g_2)}$ of $\gerbe_3$ such that the following diagram $2\mbox{-commutes}$:
  \end{sloppypar}
  \[
  \begin{tikzcd}
    \gerbe_1 \times \gerbe_2 \arrow{d}{(g_1,g_2)} \arrow{r}{F} & \gerbe_3 \arrow{d}{F_{(g_1,g_2)}} \\
    \gerbe_1 \times \gerbe_2 \arrow{r}{F} & \gerbe_3
  \end{tikzcd}
\]
\begin{sloppypar}
  The morphisms $\{F_{(g_1,g_2)}\}$ define an action of $G_1 \times G_2$ on $\gerbe_3$. Moreover, if $\kappa_1$ and $\kappa_2$ denote the obstruction classes for the actions of $G_1$ and $G_2$ acting on $\gerbe_1$ and $\gerbe_2$ respectively, then the obstruction class for the action of $G_1 \times G_2$ on $\gerbe_3$ is given by the image of $(\kappa_1,\kappa_2)$ under the natural map:
  $$ \cohomology{3}{G_1}{A(\xdown)} \times \cohomology{3}{G_2}{A(\xdown)} \cong \cohomology{3}{G_1 \times G_2}{A(\xdown) \oplus A(\xdown)} \to \cohomology{3}{G_1\times G_2}{A(\xdown)} .$$
  In particular, if $\kappa_1$ and $\kappa_2$ are zero, then the obstruction class for the action on $\gerbe_3$ is zero as well.
\end{sloppypar}
\end{lemma}
\begin{proof}
  \begin{sloppypar}
    The existence and uniqueness of the morphisms $\{F_{(g_1,g_2)}\}$ follows from the Proposition \ref{prop-hom-bands-gerbe-construction}. It follows from the uniqueness that these morphisms define an action of $G_1 \times G_2$ on $\gerbe_3$. It is straightforward to see that the obstruction class for the action of $G_1 \times G_2$  on $\gerbe_1 \times \gerbe_2$ is $(\kappa_1,\kappa_2)$. Then it follows from the Proposition \ref{prop-g-equivariant-map-and-3-cochain} that the obstruction class for the action on $\gerbe_3$ is given by the image of $(\kappa_1,\kappa_2)$ as claimed.
  \end{sloppypar}
\end{proof}
\begin{proposition}\label{prop-trivial-action-on-trivial-gerbe}
Let $\trivialgerbe$ be the category of $A\mbox{-torsors}$ which is a trivial gerbe banded by $A$. Suppose we have an action of the group $G$ on $\trivialgerbe$ such that the underlying action on $A$ is trivial. Then the associated obstruction class in $\cohomology{3}{G}{A(\xdown)}$ is zero.
\end{proposition}
\begin{proof}
  \begin{sloppypar}
    An $A\mbox{-equivariant}$ automorphism of $\trivialgerbe$ is $2\mbox{-isomorphic}$ to $\isom(y,-)$ for some global torsor $y$ on $\xdown$. In particular, these automorphisms $2\mbox{-commute}$ (See Remark \ref{remark-composition-of-morphisms-and-addition-law-on-torsors}). Therefore we can argue that the action of $G$ factors through an abelian quotient $G/N$ for some normal subgroup $N$ of $G$. By Lemma \ref{lemma-inflation-for-gerbes}, it suffices to prove that the obstruction class for the effective action by $G/N$ is zero. Since $G/N$ is abelian, we have an isomorphism ${G/N \cong G_1 \times G_2 \ldots \times G_n}$ where each $G_i$ is a cyclic group. We have a morphism of gerbes ${\bigotimes_{i = 1}^{n} \trivialgerbe \to \trivialgerbe }$ which is equivariant with respect to the addition map ${\oplus_{i = 1}^{n}A \to A}$. Then we have a component wise action of ${G_1 \times G_2 \ldots \times G_n}$ on ${\bigotimes_{i = 1}^{n} \trivialgerbe}$ where each $G_i$ acts by restriction. We know that the obstruction class for a cyclic group is zero when the underlying action is trivial (Proposition \ref{proposition-cyclic-group-trivial-action-on-A}). Therefore the claim follows from Lemma \ref{lemma-direct-product-action-on-gerbes}.
  \end{sloppypar}
\end{proof}
\begin{sloppypar}
  Suppose we are given an action of $G$ on the sheaf $A$. Let ${\invariants \xrightarrow{i} A}$ denote the invariants. We have an action of $G$ on the sheaf cohomology groups $\cohomology{1}{\xdown}{A}$ and ${\cohomology{2}{\xdown}{A}}$. In terms of torsors and gerbes, this action can be described as in Definition \ref{def-G-action-on-torsors-and-gerbes}. We want to prove that the lifts of the $G\mbox{-action}$ on the trivial gerbe (up to equivalence) are classified by the group cohomology ${\cohomology{1}{G}{\cohomology{1}{\xdown}{A}}}$. We begin with an analogous observation for torsors. Let $\trivialtorsor$ denote the sheaf $A$ which we regard as a trivial $A\mbox{-torsor}$.
\end{sloppypar}
\begin{lemma}\label{lemma-1-cocycle-for-action-on-trivial-torsor}
  \begin{sloppypar}
    Let $\{\lambda_g\}$ be automorphisms of $\trivialtorsor$ that lift the $G\mbox{-action}$. Then we have the following:
  \end{sloppypar}
  \begin{itemize}
  \item Then is a unique ${1\mbox{-cochain}}$ $\{a_g\}_{g \in G}$ with values in $A(\xdown)$ such that for $g \in G$, the map $\lambda_g$ is given by ${a \to (g \cdot a ) + a_g}$.
\item The associated $2\mbox{-cocycle}$ $\chi_{\lambda}$ is the coboundary of the $1\mbox{-cochain}$ $\{a_g\}_{g \in G}$. 
\item Let $\{a'_g\}_{g \in G}$ be another $1\mbox{-cochain}$ with values in $A(\xdown)$ and let $\{\lambda'_g\}_{g \in G}$ be the automorphisms of $\trivialtorsor$ as defined above. The two cochains $\{a_g\}$ and $\{a'_g\}$ differ by a coboundary if and only if there exists an $A\mbox{-equivariant}$ map which is $G\mbox{-equivariant}$ with respect to the two actions. 
\end{itemize}
\end{lemma}
\begin{proof}
  It is easy to verify the first part directly but it also follows from Proposition \ref{prop-hom-bands-gerbe-construction}. The second part follows from Lemma \ref{lemma-local-section-commutative-diagram-for-torsors} by considering the zero global section. For the third part, we note that an $A\mbox{-equivariant}$ automorphism of $\trivialtorsor$ is given by $a \to a + c$ for some $c \in A(\xdown)$. This map is $G\mbox{-equivariant}$ with respect to the two actions if and only if the $1\mbox{-cochains}$ $\{a_g\}_{g \in G}$ and $\{a'_g\}_{g \in G}$ differ by the coboundary of the $0\mbox{-cochain}$ $\{c\}$. 
\end{proof}  
\begin{lemma}\label{lemma-existence-of-tg-with-no-obstruction}
  \begin{sloppypar}
    There exist automorphisms ${\{T_g\}_{g \in G}}$ of the trivial gerbe $\trivialgerbe$ that lift the $G\mbox{-action}$ such that the obstruction class is zero.
  \end{sloppypar}
\end{lemma}
\begin{proof}
  \begin{sloppypar}
    The trivial gerbe $\trivialgerbe$ is induced from the trivial gerbe banded by $\invariants$. Therefore the lemma follows from Proposition  \ref{prop-induced-torsor-and-gerbes}. We can explicitly construct the morphisms ${\{T_g\}_{g \in G}}$ as follows. Suppose $(U,y)$ in $\objects(\trivialgerbe)$ is an object of $\trivialgerbe$ where $U \in \objects(\thesite)$ and $y$ is an $A|_{U}\mbox{-torsor}$ on $\thesite/U$. Then for $g \in G$, the torsor $T_g(y)$ as in Definition \ref{def-G-action-on-torsors-and-gerbes} is a torsor which as a sheaf is identical to $y$ but the action of $A|_{U}$ is twisted. Then ${T_g: \trivialgerbe \to \trivialgerbe}$ defines a morphism of gerbes which is equivariant with respect to the action of $g$ on $A$. Moreover, for $g,h \in G$, we have equality of functors $T_g \circ T_h = T_{gh}$. Therefore it follows that the obstruction class for this $G\mbox{-action}$ on $\trivialgerbe$ is zero.
  \end{sloppypar}
\end{proof}

\begin{lemma}\label{lemma-Tg-commutation-relation}
  \begin{sloppypar}
    Let $y$ be an $A\mbox{-torsor}$ on $\xdown$. Then the morphisms ${T_g \circ \isom(y,-)}$ and ${\isom(T_g(y),-) \circ T_g}$ are $2\mbox{-isomorphic}$.
  \end{sloppypar}
\end{lemma}
\begin{proof}
  We omit the details but the functor $\isom(y,-)$ (up to $2\mbox{-isomorphism}$) is locally given by ${y' \to y' \otimes y^{-1}}$. The lemma immediately follows from this observation.
\end{proof}
\begin{lemma}\label{lemma-g-action-trivial-gerbe-is-1-cocycle}
  \begin{sloppypar}
    Suppose we have automorphisms $\{F_g\}_{g \in G}$ of the trivial gerbe $\trivialgerbe$ that lift $G\mbox{-action}$ to $\trivialgerbe$. Then we have the following:
  \end{sloppypar}
  \begin{itemize}
    \begin{sloppypar}
    \item There are $A\mbox{-torsors}$ $\{y_g\}_{\in G}$ which are unique up to $A\mbox{-equivariant}$ isomorphisms such that $F_g$ and ${T_g \circ \isom(y_g,-)}$ are $2\mbox{-isomorphic}$. 
    \item The association $g \to y_g$ gives a $1\mbox{-cocycle}$ with values in ${\cohomology{1}{\xdown}{A}}$. Conversely, if $\{y'_g\}_{g \in G}$ is a $1\mbox{-cocycle}$ with values in ${\cohomology{1}{\xdown}{A}}$, then the morphisms ${\{F'_g = T_g \circ \isom(y'_g,-)\}_{g \in G}}$ define a lift of the $G\mbox{-action}$ on $\trivialgerbe$.  
    \item The two cocycles $\{y_g\}_{g \in G}$ and $\{y'_g\}_{g \in G}$ as above are cohomologous if and only if there is an $A\mbox{-equivariant}$ automorphism of $\trivialgerbe$ that is $G\mbox{-equivariant}$ with respect to the actions given by the morphisms $\{F_g\}_{g \in G}$ and $\{F'_g\}_{g \in G}$.
    \end{sloppypar}
  \end{itemize}
\end{lemma}
\begin{proof}
  We know that any $A\mbox{-equivariant}$ automorphism of $\trivialgerbe$ is $2\mbox{-isomorphic}$ to an isom functor (See Remark \ref{remark_trivialzation_given_by_global_section_for_gerbe}). Therefore by Proposition \ref{prop-hom-bands-gerbe-construction} we have the existence and uniqueness of the torsors $\{y_g\}_{g \in G}$ as claimed. The second part is straightforward verify using the commutation relation established in Lemma \ref{lemma-Tg-commutation-relation} (See Remark \ref{remark-composition-of-morphisms-and-addition-law-on-torsors}). For the third part of the lemma, we can check that the $1\mbox{-cocycles}$ $\{y_g\}$ and $\{y'_g\}$ differ by the coboundary of a $\mbox{0-cochain}$ $\{y\}$ if and only if $\isom(y,-)$ gives a $G\mbox{-equivariant}$ morphism with respect to the two actions. 
\end{proof}
\begin{sloppypar}
  Thus we see that lifts of the $G\mbox{-action}$ on $\trivialgerbe$ are parameterized by the group $\cohomology{1}{G}{\cohomology{1}{\xdown}{A}}$. We can identify certain lifts where the associated obstruction class is zero. Suppose for $g \in G$, we have an $\invariants\mbox{-torsor}$ $\overline{y}_g$ such that ${g \to \overline{y}_g}$ defines a group homomorphism ${G \to \cohomology{1}{\xdown}{\invariants}}$. Let $y_g$ denote the $A\mbox{-torsor}$ obtained by inducing $\overline{y}_g$ from $\invariants$ to $A$. Then each torsor $y_g$ represents a sheaf cohomology class in ${\cohomology{1}{\xdown}{A}^G}$. Consequently, we see that ${\{y_g\}_{g \in G}}$ is a $1\mbox{-cocycle}$ corresponding to  a cohomology class in ${\cohomology{1}{G}{\cohomology{1}{\xdown}{A}}}$. Let $F_g$ denote the automorphism ${T_g \circ \isom(y_g,-)}$ as in Lemma \ref{lemma-g-action-trivial-gerbe-is-1-cocycle}. Then the morphisms ${\{F_g\}_{g \in G}}$ define a lift of $G\mbox{-action}$ on $\trivialgerbe$. 
\end{sloppypar}
\begin{proposition}\label{prop_tg_induced_yg_has_no_obstruction}
The $G\mbox{-action}$ on $\trivialgerbe$ defined by the morphisms ${\{F_g\}_{g \in G}}$ as above has trivial obstruction class. 
\end{proposition}
\begin{proof}
  \begin{sloppypar}
    Let $\trivialaggerbe$ denote the category of ${\invariants\mbox{-torsors}}$ which is the trivial torsor banded by $\invariants$. We have an $i\mbox{-equivariant}$ morphism ${F: \trivialaggerbe \to \trivialgerbe}$ such that for all $g$ in $G$, the morphisms $F$ and $T_g \circ F$ are $2\mbox{-isomorphic}$. Since $y_g$ is induced from the $\invariants\mbox{-torsor}$ $\overline{y}_g$, the following diagram $2\mbox{-commutes}$:
  \[
    \begin{tikzcd}
      \trivialaggerbe \arrow{r}{F} \arrow{d}[swap]{\isom(\overline{y_g},-)} & \trivialgerbe \arrow{d}{\isom(y_g,-)} \\
      \trivialaggerbe \arrow{r}{F} & \trivialgerbe
    \end{tikzcd}
\]
It follows that if we define ${\overline{F_g} = \isom(\overline{y_g},-)}$, then we have a $2\mbox{-commutative}$ diagram as below:
\[
  \begin{tikzcd}
    \trivialaggerbe \arrow{r}{F} \arrow{d}[swap]{\overline{F_g}} & \trivialgerbe \arrow{d}{F_g} \\
    \trivialaggerbe \arrow{r}{F} & \trivialgerbe
  \end{tikzcd}
\]
We have an equivariant action of $G$ on $\trivialaggerbe$ given by the morphisms ${\{\overline{F_g}\}_{g \in G}}$. We note that $F$ gives a ${G\mbox{-equivariant}}$ morphism ${\trivialaggerbe \to \trivialgerbe}$. From Proposition \ref{prop-trivial-action-on-trivial-gerbe}, we know that the obstruction class for the action on $\trivialaggerbe$ vanishes. Therefore it follows that the obstruction class for the action on $\trivialgerbe$ is zero as well (See Proposition \ref{prop-g-equivariant-map-and-3-cochain}).
\end{sloppypar}
\end{proof}

}

\section{Hoschild-Serre spectral sequence}
{ 
\newcommand{\trivialtorsor}{y^{triv}}
\newcommand{\constantg}{\underline{G}}
\newcommand{\groupring}{\integers[G]}
\newcommand{\gsets}{{G}\mbox{-\textit{sh}}(\xdown_{\tau})}
\newcommand{\gmodules}{\textit{Mod}(\mathbb{Z}[\underline{G}])}
\newcommand{\topsite}{X_{top}}
\newcommand{\pline}{\projective^{1}}

\begin{sloppypar}
  A typical situation where we have a sheaf $A$ with a group action occurs in the following scenario. Suppose $\xdown$ is a scheme and we consider the small \etale{} site of $\xdown$. Let ${\xup \xrightarrow{\pi}\xdown}$ be a Galois cover with the Galois group $G$. Then for any sheaf $E$ on $\xup$, the sheaf $\pi_*(E)$ carries a natural $G$ action. Moreover, the higher group cohomology locally vanishes for all indices for this action. In this section, we mimic this scenario in a more general setting of an arbitrary site. 

We denote by $\constantg$ the constant group sheaf with value $G$ on $\thesite$. We denote by $\gsets$ the category whose objects are sheaves with an action of $G$ and the morphisms are ${G\mbox{-equivariant}}$ morphisms. Similarly, we denote by $\gmodules$ the abelian category whose objects are abelian sheaves with an action of $G$ and the morphisms are ${G\mbox{-equivariant}}$ morphisms of abelian sheaves. Let $M$ be a left ${\constantg\mbox{-torsor}}$ (Definition \ref{definition-a-torsor}). We have an action of $G$ on $M$ which is given via the sheafification map ${G \to \constantg}$.
\end{sloppypar}
\begin{remark}\label{remark-constangg-vs-g-sheaves-confusion}
Suppose $B$ is a sheaf with an action of $G$. Thus we have a map of presheaves
${G \times B \to B}$. Since sheafification is exact, we have a bijection:
$$ \morphisms_{\presheaves(\thesite)}(G \times B,B) \cong \morphisms_{\sheaves(\thesite)}(\constantg \times B, B) .$$
Thus we see that an action of $G$ induces an action of $\constantg$. We can define a category whose objects are sheaves with an action of $\constantg$ and morphisms are $\constantg\mbox{-equivariant}$ morphisms. This category is equivalent to $\gsets$. 
\end{remark}
\begin{definition}
Suppose $E$ is a sheaf on $\thesite$. We define $E[M]$ as the internal hom sheaf (See \cite[\href{https://stacks.math.columbia.edu/tag/04TP}{Tag 04TP}]{stacks-project}) which assigns to every $U \in \objects(\thesite)$, the set ${\morphisms_{Sh(\xdown_{\tau}/U)}(M|_{U},E|_{U})}$. We have a left action of $G$ on $E[M]$ which is locally defined by ${(g \cdot f)(m) = f( g^{-1} \cdot m)}$. When $E$ is an abelian sheaf, so is $E[M]$. The abelian group structure on $E[M]$ is given via the evaluation map to $E$. It is clear that in this case the group $G$ acts by group homomorphisms.
\end{definition}
\begin{lemma}\label{lemma-when-M-is-trivial-torsor}
Suppose $M \cong \constantg$ is the trivial torsor. Then we have an isomorphism of sheaves ${E[M] \cong \prod_{g \in G} E}$. Moreover, if we let $G$ act on the right hand side by permuting the indices, then the isomorphism is $G\mbox{-equivariant}$. 
\end{lemma}
\begin{proof}
  \begin{sloppypar}
    By the universal property of sheafification, for $U \in \objects(\thesite)$, we have
    $$ \morphisms_{\mathit{Psh}(\xdown_{et})}(G|_{U},E|_{U}) \cong \morphisms_{\sheaves(\xdown_{et})}(\constantg|_{U},E|_U) $$
    where by $G|_{U}$ we mean the constant presheaf with value $G$ on $\thesite/U$. Clearly, the left hand side is equal to ${\prod_{g \in G} E|_U}$. These identifications glue to give the desired isomorphism as claimed.
  \end{sloppypar}
\end{proof}
\begin{corollary}\label{corollary-local-group-cohomology-vanishes}
 Let $E$ be an abelian sheaf on $\thesite$. Then the group cohomology locally vanishes for all indices $j > 0$ for the action of $G$ on $E[M]$. 
\end{corollary}
\begin{proof}
By the above Lemma \ref{lemma-when-M-is-trivial-torsor}, for $U \in \objects(\thesite)$, we have a cover $\{U_i \to U\}_{i \in I}$ such that $E[M]|_{U_i}$ is isomorphic to ${\prod_{g \in G} {E}|_{U_i}}$ where $G$ acts by permuting the indices. Therefore for all $j > 0$, the cohomology groups $\cohomology{j}{G}{E[M](U_i)}$ vanish. 
\end{proof}
\begin{lemma}\label{lemma-e[m]-is-exact-functor}
  \begin{sloppypar} The functor ${\absheaves(\thesite) \to \gmodules}$ defined by ${E \to E[M]}$ is an exact functor.
  \end{sloppypar}
\end{lemma}
\begin{proof}
  \begin{sloppypar}
    We note that a complex in $\gmodules$ is exact if and only if the underlying complex of abelian sheaves is exact. If $M$ is a trivial torsor then we have an isomorphism ${E[M] \cong \prod_{g \in G} E}$ which is functorial in $E$. Since $M$ is locally trivial, we conclude that the functor $E \to E[M]$ is locally exact and therefore exact.
  \end{sloppypar}
\end{proof}
\begin{lemma}\label{lemma-invariants-of-E[M]}
  \begin{sloppypar}
    We have an injective morphism of sheaves $E \to E[M]$ that identifies $E$ as the sheaf of invariants of $E[M]$ under the action of $G$.
  \end{sloppypar}
\end{lemma}
\begin{proof}
  \begin{sloppypar}
    For an object $U$ of $\thesite$, we can define a map ${E(U) \to \morphisms_{Sh(\thesite/U)}(M|_{U},E|_{U})}$ which sends an element $e$ to the morphism which on each object is a constant map with the value equal to the restriction of $e$. Thus we have a morphism ${E \to E[M]}$. We know $E[M]$ is locally isomorphic to $\prod_{g \in G} E$ on which $G$ acts by permuting the indices. Thus we see that this map is locally injective and isomorphic onto the fixed points of $E[M]$.
    \end{sloppypar}
  \end{proof}
  
  \begin{sloppypar}
    The functor ``internal morphisms to M'' is right adjoint to the functor ${- \times M}$ (See \cite[\href{https://stacks.math.columbia.edu/tag/0BWQ}{Tag 0BWQ}]{stacks-project}). That is, for a sheaf $B$ on $\thesite$, we have a bijection
    $$ \morphisms_{\sheaves(\thesite)}(B,E[M]) \cong \morphisms_{\sheaves(\thesite)}(B \times M, E)$$
 which is functorial in all three entrees. Suppose $B$ is a sheaf on which we have a left action of the group $G$. As in Remark \ref{remark-constangg-vs-g-sheaves-confusion}, we consider the associated action of $\constantg$ on $B$. For $U \in \objects(\thesite)$, we define an equivalence relation on $B(U) \times M(U)$ by declaring ${(b,m) \sim \left(g \cdot b ,g \cdot m\right)}$ for all ${g \in \constantg(U)}$.
    It is clear that the equivalence classes modulo this relation define a presheaf on $\thesite$. We denote this presheaf by $\sfrac{B \times M }{\sim}$ and we define $B \times_{G} M$ as the sheaffication of this presheaf. 
  \end{sloppypar}
  \begin{lemma}\label{lemma-right-adjoint-to-G-sheaves}
  The functor
  \begin{align*}
    - \times_{G} M : \gsets & \to \sheaves(\thesite) \\
                              B & \to B \times_{G} M
  \end{align*}
is left adjoint to the functor
  \begin{align*}
    -[M] : \sheaves(\thesite) & \to \gsets \\
                              E & \to E[M]
  \end{align*}
\end{lemma}
\begin{proof}
  \begin{sloppypar}
    In other words, we want to prove that there is a bijection 
    $$ \morphisms_{\gsets}(B,E[M]) \cong \morphisms_{\sheaves(\thesite)}(B \times_{G} M,E) $$  
    which is functorial with respect to $B$ and $E$. We have a bijection
    \begin{align*}
      \morphisms_{\sheaves}(B,E[M]) &\cong  \morphisms_{\sheaves}(B \times M,E) \end{align*}
    which we denote by $f \to \widetilde{f}$. For $U \in \objects(\thesite)$, and ${b \in B(U)}$, we have the morphism ${f(b): M|_U \to E|_U}$ of sheaves over $\thesite/U$. The morphism $\widetilde{f}$ on $U$ is given by 
    \begin{align*}
      \widetilde{f}: B(U)\times M(U) & \to E(U) \\
      (b,m) & \to (f(b))(m)
    \end{align*}
    By Remark \ref{remark-constangg-vs-g-sheaves-confusion}, we note that $f$ is $G\mbox{-equivariant}$ if and only if it is equivariant with respect to the action of $\constantg$. If this is the case, then we see that for all ${(b,m) \in B(U) \times M(U)}$ and $g \in \constantg(U)$, we have ${\widetilde{f}(b,m) =  \widetilde{f}(g \cdot b,g \cdot m)}$. Since ${B \times M \to \sfrac{B \times M}{\sim}}$ is an epimorphism of presheaves, we see that $\widetilde{f}$ uniquely factors through ${\sfrac{B \times M}{\sim}}$. By the universal property of sheafification, it factors uniquely through the sheaf $B \times_{G} M$. Thus we get an injective map
    $$ \morphisms_{\gsets}(B,E[M]) \to \morphisms_{\sheaves(\thesite)}(B \times_{G} M,E) $$ 
By following the same argument backward, we see that it is surjective as well.
  \end{sloppypar}
\end{proof}
\begin{sloppypar}
  Suppose $B$ is an abelian sheaf on $\thesite$ on which $G$ acts by group homomorphisms. Then the sheaf $B \times_{G} M$ also has a structure of an abelian sheaf. This can be seen as follows. Suppose $U$ is an object of $\thesite$ and we have ${(b_1,m_1),(b_2,m_2) \in B(U) \times M(U)}$. For any $m_3 \in M(U)$, there are unique elements $g_1,g_2 \in \constantg(U)$ such that ${g_1 \cdot m_1 = m_3}$ and ${g_2 \cdot m_2 = m_3}$. Then we define 
  $$ (b_1,m_1) + (b_2,m_2) = (g_1 \cdot b_1 + g_2 \cdot b_2,m_3).$$
Thus we see that $\sfrac{B \times M}{\sim}$ is an abelian  presheaf. After sheafification, we get a structure of an abelian sheaf on $B \times_{G} M.$ 
\end{sloppypar}
  \begin{remark} \label{remark-on-construction-of-b-x-m-mod-G}
    We can interpret the construction of $B \times_{G} M$ as follows. A trivialization for $M$ is a cover of $\xdown$ where $M$ is trivial together with gluing morphisms which are elements of the group $G$ (such that the cocycle condition is met). The sheaf $B \times_{G} M$ is locally identical to $B$ wherever $M$ is trivial and the gluing data is borrowed from the gluing data of $M$ via the homomorphism of $G$ to $\text{Aut}(B)$.
  \end{remark}
\begin{lemma}\label{lemma-abelian-sheaves-adjoint}
  The functor
  \begin{align*}
    \gmodules & \to \absheaves(\thesite) \\
    B &\to B \times_{G} M
  \end{align*}
is left adjoint to the functor
\begin{align*}
   \absheaves(\thesite) & \to \gmodules \\
  E &\to E[M]
\end{align*}
\end{lemma}
\begin{proof}
  By Lemma \ref{lemma-right-adjoint-to-G-sheaves} we have a bijection
$$ \morphisms_{\gsets}(B,E[M]) \cong \morphisms_{\sheaves(\thesite)}(B \times_{G} M,E) .$$  
It is clear from the proof of Lemma \ref{lemma-abelian-sheaves-adjoint} that when $B$ and $E$ are abelian sheaves the above correspondence gives a bijection
$$ \morphisms_{\gmodules}(B,E[M]) \cong \morphisms_{\absheaves(\thesite)}(B \times_{G} M,E).$$
\end{proof}
\begin{lemma}\label{lemma-b-times-g-m-when-m-is-trivial}
  Suppose $M \cong \constantg$ is the trivial torsor. Then for any module $B$ in $\gmodules$ we have $B \times_{G} M \cong B$. 
\end{lemma}
\begin{proof}
  \begin{sloppypar}
When $M$ is trivial, we have ${E[M] \cong \prod_{g \in G} E}$. We see that we have a functorial bijection:
    $$ \morphisms_{\gmodules}(B, \prod_{g \in G} E ) \cong \morphisms_{\absheaves(\thesite)}(B,E) .$$
    Therefore the conclusion follows from Yoneda's lemma and the above Lemma \ref{lemma-abelian-sheaves-adjoint} (See Remark \ref{remark-on-construction-of-b-x-m-mod-G}).
  \end{sloppypar}
\end{proof}
\begin{lemma}\label{lemma-b-times-m-mod-g-functor-is-left-exact}
  The functor
  \begin{align*}
    \gmodules &\to \absheaves(\thesite) \\
    B &\to B \times_{G} M                           
  \end{align*}
  is exact. 
\end{lemma}
\begin{proof}
Suppose ${0 \to B_1 \to B_2 \to B_3 \to 0 }$ is an exact sequence in $\gmodules$. Since $M$ is locally trivial, it follows from the above Lemma \ref{lemma-b-times-g-m-when-m-is-trivial} that for $U \in \objects(\thesite)$, there exists a cover $\{U_i \to U\}_{i \in I}$ such that the sequence
$$ 0 \to B_1 \times_{G} M  \to B_2 \times_{G} M \to B_2 \times_{G} M \to 0 $$
is exact on $U_i$. This is sufficient to conclude exactness of sheaves.
\end{proof}
\begin{lemma}\label{lemma-injective[m]-is-acyclec}
  \begin{sloppypar}
    For an injective abelian sheaf $I$ on $\thesite$, the group cohomologies ${\cohomology{j}{G}{\cohomology{0}{\xdown}{I[M]}}}$ vanish for $j > 0$.
  \end{sloppypar}
\end{lemma}
\begin{proof}
It suffices to prove that $\cohomology{0}{\xdown}{I[M]}$ is an injective module over the group ring $\integers[G]$. We know that the functor
\begin{align*}
  H^0(\xdown, - ) : \gmodules &\to \textit{Mod}(\integers[G])
\end{align*}
admits a left adjoint given by ${L \to \underline{L}}$. Moreover, this left adjoint is exact. Therefore it follows from the above Lemma \ref{lemma-abelian-sheaves-adjoint} that the composite functor
  \[
    \begin{tikzcd}[row sep=tiny]
      \absheaves(\thesite) \arrow{r} & \gmodules  \arrow{r} & \textit{Mod}(\integers[G]) \\
      E \arrow{r} &  E[M] \arrow{r} & H^0(\xdown,-) 
    \end{tikzcd}
  \]
has a left adjoint given by the composite functor 
\[
  \begin{tikzcd}[row sep=tiny]
    \textit{Mod}(\integers[G]) \arrow{r} & \gmodules  \arrow{r} & \absheaves(\thesite) \\
    L \arrow{r} &  \underline{L} \arrow{r} & \underline{L} \times_{G} M.
  \end{tikzcd}
\]
\begin{sloppypar}
  Moreover, by Lemma \ref{lemma-b-times-m-mod-g-functor-is-left-exact} we  conclude that the composite left adjoint is exact. It follows from formal adjointness property that if a functor has an exact left adjoint, then it transforms injective objects to injective objects  (See \cite[\href{https://stacks.math.columbia.edu/tag/015Z}{Tag 015Z}]{stacks-project}). Therefore ${\cohomology{0}{\xdown}{I[M]}}$ is an injective $\integers[G]$ module.
\end{sloppypar}
\end{proof}
\begin{theorem}\label{theorem-existence-of-spectral-sequence}
  \begin{sloppypar}
    Let $E$ be an abelian sheaf on $\thesite$. Then there is a spectral sequence that starts with ${E_2^{p,q} = \cohomology{q}{G}{\cohomology{p}{\xdown}{E[M]}}}$ and converges to $\cohomology{p+q}{\xdown}{E}$. 
  \end{sloppypar}
\end{theorem}
\begin{proof}
  \begin{sloppypar}
    By Lemma \ref{lemma-invariants-of-E[M]} we see that the functor of global sections ${\absheaves(\thesite) \to \absheaves}$ defined by ${E \to \cohomology{0}{\xdown}{E}}$, is isomorphic to the composite functor:
    \[
      \begin{tikzcd}[row sep = tiny]
        \absheaves(\thesite) \arrow{r} & \textit{Mod}(\integers[G]) \arrow{r} & \absheaves \\
        E \arrow{r} & \cohomology{0}{\xdown}{E[M]} \arrow{r} & (\cohomology{0}{\xdown}{E[M]})^{G}. 
      \end{tikzcd}
    \]
    Moreover, for an injective abelian sheaf $I$, the module $\cohomology{0}{\xdown}{I[M]}$ has no higher group cohomologies. Therefore we have the Grothendieck spectral sequence that computes the sheaf cohomologies of $E$. Since $E \to E[M]$ is an exact functor (Lemma \ref{lemma-e[m]-is-exact-functor}), we can identify the right derived functors of ${E \to \cohomology{0}{\xdown}{E[M]}}$ with ${\cohomology{p}{\xdown}{E[M]}}$. Therefore the spectral sequence starts with ${E_2^{p,q} = \cohomology{q}{G}{\cohomology{p}{\xdown}{E[M]}}}$ and converges to $\cohomology{p+q}{\xdown}{E}$.
  \end{sloppypar}
\end{proof}
\subsection*{Hochschild-Serre spectral sequence}
\begin{sloppypar}
  Suppose $\xdown$ is a scheme and $\etalesite$ is the small \etale{} site. Then any $G\mbox{-torsor}$ $M$ on $\etalesite$ is representable (See \cite[\href{https://stacks.math.columbia.edu/tag/09Y8}{Tag 09Y8}]{stacks-project}). That is, there is a finite \etale{} morphism $\xup \xrightarrow{\pi} \xdown$ together with an isomorphism $G \cong \aut(\xup / \xdown)$ such that ${M \cong \morphisms_{\etalesite}(-,\xup)}$. In this situation, we have the Hochschild-Serre spectral sequence with ${E_2^{p,q} = \cohomology{q}{G}{\cohomology{p}{\xup}{\pi^{-1}E}}}$ that converges to $\cohomology{p+q}{\xdown}{E}$. We can realize this as another instance of the Grothendieck spectral sequence given by writing the global sections functor as a composition of two functors:
\end{sloppypar}
\[
\begin{tikzcd}[row sep = tiny]
  \absheaves(\thesite) \arrow{r} & \textit{Mod}(\integers[G]) \arrow{r} & \absheaves \\
  E \arrow{r} & \cohomology{0}{\xup}{\pi^{-1}E} \arrow{r} & (\cohomology{0}{\xup}{\pi^{-1}E})^{G}. 
\end{tikzcd}
\]
When the torsor $M$ is representable by the object $\xup \xrightarrow{\pi} \xdown$, the internal sheaf hom $E[M]$ is isomorphic to $ \pi_*(E|_{\xup}) $ (See \cite[\href{https://stacks.math.columbia.edu/tag/0D7X}{Tag 0D7X}]{stacks-project}). Thus we have
$$ \cohomology{0}{\xdown}{E[M]} \cong \cohomology{0}{\xdown}{\pi_*(E|_{\xup})} \cong \cohomology{0}{\xup}{\pi^{-1}E} .$$
Therefore in this situation, the spectral sequence in Theorem \ref{theorem-existence-of-spectral-sequence} is the same as the Hochschild-Serre spectral sequence.
\subsection*{Long exact sequence in lower degrees} 
The second page of the spectral sequence in Theorem \ref{theorem-existence-of-spectral-sequence} is as below:
\[
  \begin{tikzcd}[cramped,sep=small]
    \cohomology{2}{\xdown}{E[M]}^{G}  & \phantom{} & \phantom{} & \phantom{} \\
\cohomology{1}{\xdown}{E[M]}^{G} \arrow{rrd} & \cohomology{1}{G}{\cohomology{1}{\xdown}{E[M]}} \arrow{rrd} & \cohomology{2}{G}{\cohomology{1}{\xdown}{E[M]}} & \phantom{} \\
    \cohomology{0}{G}{E[M]}^{G} & \cohomology{1}{G}{\cohomology{0}{G}{E[M]}} &  \cohomology{2}{G}{\cohomology{0}{G}{E[M]}} & \cohomology{3}{G}{\cohomology{0}{G}{E[M]}} \\
  \end{tikzcd}
\]
Our goal is to describe a long exact sequence arising from this spectral sequence in terms of torsors and gerbes. We write the sheaf $E[M]$ as $A$. We identify $E$ as the sheaf of invariants $A^G$ and we denote by $i: \invariants \to A$ the natural inclusion map. With this notation we have the following long exact sequence:
\begin{figure}[H]
  \centering
  \[
    \begin{tikzcd}
      0 \arrow{r} & \cohomology{1}{G}{A(\xdown)} \ar[draw=none]{d}[name=Y, anchor=center]{} \arrow{r}{\theta_1} & \cohomology{1}{\xdown}{A^G} \ar[rounded corners,
      to path={ -- ([xshift=2ex]\tikztostart.east)
        |- (Y.center) \tikztonodes
        -| ([xshift=-2ex]\tikztotarget.west)
        -- (\tikztotarget)}]{dll}[swap, at end]{\theta_2} \\ 
      \cohomology{1}{\xdown}{A}^G \arrow{r}{\theta_3} & \cohomology{2}{G}{A(\xdown)} \arrow{r}{\theta_4} \ar[draw=none]{d}[name=X, anchor=center]{} & \left(\text{ker:}\cohomology{2}{\xdown}{A^G} \to  \cohomology{2}{\xdown}{A}^G \right) \ar[rounded corners,
      to path={ -- ([xshift=2ex]\tikztostart.east)
        |- (X.center) \tikztonodes
        -| ([xshift=-2ex]\tikztotarget.west)
        -- (\tikztotarget)}]{dll}[swap, at end]{\theta_5} \\
      \cohomology{1}{G}{\cohomology{1}{\xdown}{A}} \arrow{r}{\theta_6} & \cohomology{3}{G}{A(\xdown)}
    \end{tikzcd}
  \]
\stepcounter{theorem}
\caption{Long exact sequence of cohomologies}
\label{figure-long-exact-sequence}
\end{figure}
\begin{sloppypar}
  We use the following dictionary to concretely identify the various terms which appear above (See Theorem \ref{theorem-cohomologies-and-torsors-gerbes}). 
  \begin{itemize}
  \item $\cohomology{1}{\xdown}{\invariants} \cong \left\{ \text{ }\invariants\mbox{-torsors} \text{ up to isomorphism }\right\}$
  \item $\cohomology{2}{\xdown}{\invariants} \cong \left\{ \text{ Gerbes banded by }\invariants \text{ up to isomorphism }\right\}$
  \item $\cohomology{1}{\xdown}{A} \cong \left\{ \text{ }A\mbox{-torsors} \text{ up to isomorphism }\right\}$
  \item $\cohomology{2}{\xdown}{A} \cong \left\{ \text{ Gerbes banded by } A \text{ up to isomorphism }\right\}$
  \end{itemize}
  The action of the group $G$ on ${{\cohomology{1}{\xdown}{A}}}$ and ${{\cohomology{2}{\xdown}{A}}}$ in terms of torsors and gerbes is as given by Definition \ref{def-G-action-on-torsors-and-gerbes}. The maps 
  ${\cohomology{1}{\xdown}{\invariants} \to \cohomology{1}{\xdown}{A}}$ and $\cohomology{2}{\xdown}{\invariants} \to \cohomology{2}{\xdown}{A}$ are given by inducing a torsor or a gerbe from the invariants (See Definition \ref{def-induced-torsor-and-gerbe}). It follows from Proposition \ref{prop-induced-torsor-and-gerbes} that a torsor or a gerbe induced from the invariants admits a lift of the group action by $G$ and therefore represents a cohomology class in ${{\cohomology{1}{\xdown}{A}}^G}$ or ${{\cohomology{2}{\xdown}{A}}^G}$ (See Remark \ref{remark-tg-operators-define-action-on-sheaf-cohomology}). We note that higher group cohomology locally vanishes for all indices for the $G\mbox{-action}$ on $A$ (Corollary \ref{corollary-local-group-cohomology-vanishes}).

We denote the trivial (See Definition \ref{definition-a-torsor}) $A\mbox{-torsor}$ as $\trivialtorsor$. We denote by $\trivialgerbe$ the category $\torsorcat{A}$ which is a trivial gerbe banded by $A$. Now we begin to define the six maps that appear in the Figure \ref{figure-long-exact-sequence}.

  Suppose $\{a_g\}_{g \in G}$ is $1\mbox{-cocycle}$ with values in $A(\xdown)$ representing a cohomology class in $\cohomology{1}{G}{A(\xdown)}$. Then we have a lift of the group action on $\trivialtorsor$ given by the morphisms $\{\lambda_g\}_{g \in G}$ as defined in Lemma \ref{lemma-1-cocycle-for-action-on-trivial-torsor}. Then the associated cochain $\chi_{\lambda}$ is identically zero. Since the first group cohomology classes locally vanish, $\trivialtorsor_{\lambda}$ is an $\invariants\mbox{-torsor}$. Moreover, we can deduce from Lemma \ref{lemma-1-cocycle-for-action-on-trivial-torsor} and Remark \ref{remark-functoriality-of-y-alpha-gerbe-f-constructions} that the $\invariants\mbox{-equivariant}$ isomorphism class of $\trivialtorsor_{\lambda}$ only depends on the cohomology class of the cocycle $\{a_g\}_{g \in G}$. It is easy to verify (also follows from Lemma \ref{lemma-local-cohomology-obstruction-for-local-sections-of-y-lambda}) that the torsor $\trivialtorsor_{\lambda}$ admits a global section if and only if the group cohomology class of $\{a_g\}_{g \in G}$ is zero. Therefore this construction defines an injective map $\cohomology{1}{G}{A(\xdown)} \xrightarrow{\theta_1} \cohomology{1}{\xdown}{\invariants}$ that begins the long exact sequence in the Figure \ref{figure-long-exact-sequence}.

  As we remarked earlier, the map ${\cohomology{1}{\xdown}{A^G} \xrightarrow{\theta_2} \cohomology{1}{\xdown}{A}^G}$ is given by inducing an $\invariants\mbox{-torsor}$ to $A$. It follows from Theorem \ref{theorem-induced-from-invariants-when-cohomology-class-is-zero} and Proposition \ref{prop-induced-torsor-and-gerbes} that the kernel of $\theta_2$ is precisely the image of $\theta_1$.

  Suppose $y$ is an $A\mbox{-torsor}$ which is stable under the action of $G$. In other words, we have a lift of the $G\mbox{-action}$ on $y$. The associated obstruction cohomology class in $\cohomology{2}{G}{A(\xdown)}$ is independent of the choice of such a lift. Thus we have a map ${\cohomology{1}{\xdown}{A}^{G}\xrightarrow{\theta_3}\cohomology{2}{G}{A(\xdown)}}$. It follows from Theorem \ref{theorem-induced-from-invariants-when-cohomology-class-is-zero} and Proposition \ref{prop-induced-torsor-and-gerbes} that the kernel of this map consists of precisely those $A\mbox{-torsors}$ which are induced from the invariants. 

  The next three maps in the long exact sequence involve the second cohomology and therefore gerbes come into play. The trivial gerbe $\trivialgerbe$ admits a lift of the $G\mbox{-action}$ given by the morphisms $\{T_g\}_{g \in G}$ as in Lemma \ref{lemma-existence-of-tg-with-no-obstruction}. For $g,h \in G$, we have equality of morphisms ${T_{gh} = T_g \circ T_h}$. Therefore the data of connecting $2\mbox{-morphisms}$ for this action corresponds to an automorphism ${\alpha_{g,h}: T_{gh} \to T_{gh}}$ for every $g,h \in G$. We regard this data as a $2\mbox{-cochain}$ with values in $A(\xdown)$. Suppose we are given such a $2\mbox{-cochain}$ which is a cocycle representing a cohomology class in $\cohomology{2}{G}{A(\xdown)}$. 
\begin{lemma}\label{lemma-3-cochain-for-T-g-functors}
The associated $3\mbox{-cocycle}$ $\kappa_{(T,\alpha)}$ for the above group action is zero.
\end{lemma}
\begin{proof}
We have the global section $\trivialtorsor$ of $\trivialgerbe$ for which the map ${g^{-1}: \trivialtorsor \xrightarrow{a \to g^{-1} \cdot a} \trivialtorsor}$ gives an isomorphism of torsors $\trivialtorsor \to T_g(\trivialtorsor)$. Thus we have a scenario as in Situation \ref{situation-for-local-objects-for-gerbe-F-alpha}, where we can apply Lemma \ref{lemma-locally-kappa-f-alpha-is-coboundary} to see that the $3\mbox{-cochain}$ $\kappa_{(T,\alpha)}$ is the coboundary of the $2\mbox{-cochain}$ $\{\alpha_{g,h}\}_{g,h \in G}$. Since $\{\alpha_{g,h}\}_{g,h \in G}$ is a cocycle, this proves the lemma. 
\end{proof}
Thus we conclude that $\trivialgerbe_{(T,\alpha)}$ is a gerbe banded by $\invariants$ which when induced to $A$ gives the trivial gerbe. Suppose $\{\alpha'_{g,h}\}_{g,h \in G}$ is another cochain which is cohomologous to $\{\alpha_{g,h}\}_{g,h \in G}$. Then for every $g \in G$, there is a $2\mbox{-automorphism}$ ${\gamma_g: T_g \to T_g}$ given by a scalar $\gamma_g \in A(\xdown)$ such that Corollary \ref{corollary-change-connecting-morphisms-gives-same-gerbe-f-alpha} yields an isomorphism of gerbes ${\gerbe_{(T,\alpha)} \cong \gerbe_{(T,\alpha')}}$. Thus we see that the assignment ${\{\alpha_{g,h}\}_{g,h \in G} \to \gerbe_{(T,\alpha)}}$ gives a well-defined map:
$$ \cohomology{2}{G}{A(\xdown)} \xrightarrow{\theta_4} \left(\text{ker:}\cohomology{2}{\xdown}{A^G} \to  \cohomology{2}{\xdown}{A}^G \right) .$$

  By unwinding the definitions, we see that a global section of $\gerbe_{(T,\alpha)}$ is a tuple $(y,\{\phi_g\}_{g \in G})$ where $y$ is an $A\mbox{-torsor}$ and $\{\phi_g: y \to y\}_{g \in G}$  are automorphisms of $y$ which lift the $G\mbox{-action}$ such that the associated $2\mbox{-cochain}$ is $\{\alpha_{g,h}\}_{g,h \in G}$. Thus we see that the kernel of the map $\theta_4$ is precisely the image of $\theta_3$.

Suppose $\overline{\gerbe}$ is a gerbe banded by $\invariants$ such that we have an $i\mbox{-equivariant}$ morphism ${F: \overline{\gerbe} \to \trivialgerbe}$. By Proposition \ref{prop-induced-torsor-and-gerbes}, we have automorphisms ${\{F_g\}_{g \in G}}$ of $\trivialgerbe$ that lift the $G\mbox{-action}$ such that the $F$ and $F_g \circ F$ are $2\mbox{-isomorphic}$. By Lemma \ref{lemma-g-action-trivial-gerbe-is-1-cocycle}, we have a $1\mbox{-cocycle}$ $\{y_g\}_{g \in G}$ with values in $\cohomology{1}{\xdown}{A}$ such that for $g \in G$, the automorphism $F_g$ is $2\mbox{-isomorphic}$ to ${T_g \circ \isom(y_g,-)}$.
\begin{lemma}\label{lemma-theta5-well-defined}
  \begin{sloppypar}
    The cohomology class of $\{y_g\}_{g \in G}$ in ${\cohomology{1}{G}{\cohomology{1}{\xdown}{A}}}$ only depends on the $\invariants\mbox{-equivariant}$ isomorphism class of the gerbe $\overline{\gerbe}$.
  \end{sloppypar}
\end{lemma}
\begin{proof}
  \begin{sloppypar}
    It suffices to prove that the cohomology class of $\{y_g\}_{g \in G}$ is independent of the choice of the ${i\mbox{-equivariant}}$ morphism $F: \overline{\gerbe} \to \trivialgerbe$. Suppose $F': \overline{\gerbe} \to \trivialgerbe $ is another $i\mbox{-equivariant}$ morphism. Let $\{F'_g\}_{g \in G}$ be the automorphisms of $\trivialgerbe$ such that for $g \in G$, the morphisms $F'$ and ${F'_g \circ F'}$ are ${2\mbox{-isomorphic}}$. By Proposition \ref{prop-hom-bands-gerbe-construction} there is a unique (up to $2\mbox{-isomorphism}$) $A\mbox{-equivariant}$ isomorphism ${H: \trivialgerbe \to \trivialgerbe}$ such that the following diagram ${2\mbox{-commutes}}$:
  \end{sloppypar}
  \[
  \begin{tikzcd}
    \overline{\gerbe} \arrow{r}{F} \arrow{d}{\rotatebox{90}{=}} & \trivialgerbe \arrow{d}{H} \\
    \overline{\gerbe} \arrow{r}{F'} & \trivialgerbe
  \end{tikzcd}
\]
\begin{sloppypar}
  It is straightforward to verify that $H$ defines a $G\mbox{-equivariant}$ morphism with respect to the actions given by the morphisms $\{F_g\}_{g \in G}$ and $\{F'_g\}_{g \in G}$. By Lemma \ref{lemma-g-action-trivial-gerbe-is-1-cocycle}, we see that the cohomology class in ${\cohomology{1}{G}{\cohomology{1}{\xdown}{A}}}$ associated to the two actions is the same.
\end{sloppypar}
\end{proof}
Thus we have a well-defined map:
$$\left(\text{ker:}\cohomology{2}{\xdown}{A^G} \to  \cohomology{2}{\xdown}{A}^G \right) \xrightarrow{\theta_5} \cohomology{1}{G}{\cohomology{1}{\xdown}{A}}.$$
A gerbe $\overline{\gerbe}$ lies in the kernel of the map $\theta_5$ if and only if there is an $i\mbox{-equivariant}$ morphism $\overline{\gerbe} \xrightarrow{F} \trivialgerbe$ such that for $g \in G$, the morphisms $F$ and $T_g\circ F$ are $2\mbox{-isomorphic}$. By Theorem \ref{theorem-induced-from-invariants-when-cohomology-class-is-zero}, we see that this holds if and only if $\overline{\gerbe}$ is in the image of $\theta_4$. Thus we see that the map $\theta_5$ fits in the long exact sequence in Figure \ref{figure-long-exact-sequence}. 

Suppose we have a $1\mbox{-cocycle}$ $\{y_g\}_{g \in G}$ with values in $\cohomology{1}{\xdown}{A}$. By Lemma \ref{lemma-g-action-trivial-gerbe-is-1-cocycle}, we have morphisms ${\{F_g = T_g \circ \isom(y_g,-)\}_{g \in G}}$ that define a lift of the action of $G$ on $\trivialgerbe$. Then we have the associated obstruction cohomology class in ${\cohomology{3}{G}{A(\xdown)}}$. By Lemma \ref{lemma-g-action-trivial-gerbe-is-1-cocycle} and functoriality of the obstruction class (Proposition \ref{prop-g-equivariant-map-and-3-cochain}) we see that the obstruction class only depends on the cohomology class of ${\{y_g\}_{g \in G}}$ in ${\cohomology{1}{G}{\cohomology{1}{\xdown}{A}}}$. Thus we have a well-defined map ${\cohomology{1}{G}{\cohomology{1}{\xdown}{A}} \xrightarrow{\theta_6} \cohomology{3}{G}{A(\xdown)}}$. We conclude from Proposition \ref{prop-induced-torsor-and-gerbes} that the image of $\theta_5$ lies in the kernel of the map $\theta_6$. Conversely, if the associated cohomology class in $\cohomology{3}{G}{A(\xdown)}$ is zero then by Theorem \ref{theorem-induced-from-invariants-when-cohomology-class-is-zero}, the corresponding element in $\cohomology{1}{G}{\cohomology{1}{\xdown}{A}}$ lies in the image of $\theta_5$.
\end{sloppypar}

\begin{proposition}\label{prop-group-hom-are-killed-by-theta-6}
The composite morphism 
$$ \hom(G,\cohomology{1}{\xdown}{\invariants}) \to \cohomology{1}{G}{\cohomology{1}{\xdown}{A}} \xrightarrow{\theta_6} \cohomology{3}{G}{A(\xdown)} $$
is zero. 
\end{proposition}
\begin{proof}
 This follows immediately from Proposition \ref{prop_tg_induced_yg_has_no_obstruction}. 
\end{proof}

\section{Application to Galois covers}
{
  
\begin{sloppypar}
Let $\xdown$ be a quasi-compact scheme or a quasi-compact Deligne-Mumford stack. Let $\gmsheaf$ denote the sheaf of multiplicative group on $\etalesite$. Let ${\pi: \xup \to \xdown}$ be a finite Galois \etale{} cover with Galois group $G$. Let ${\{\tau_g\}_{g \in G}}$ denote the automorphisms of $\xup$ over $\xdown$. 
\end{sloppypar}
\begin{proposition}\label{prop_galois_cover_line_bundle_exists}
  \begin{sloppypar}
    Assume that at least one of the following holds:
    \begin{enumerate}
    \item $\etalecohomology{2}{\xdown}{\gmsheaf}$ = 0.
    \item $\etalecohomology{2}{\xdown}{\gmsheaf}$ has no $n\mbox{-torsion}$ where $n$ is the order of $G$ and $\cohomology{2}{G}{\gmsheaf(\xup)} = 0$.
    \end{enumerate}
    Let ${\Phi: G \to \pic(\xdown)}$ be a group homomorphism. Then there exists a line bundle $L$ on $\xup$ such that for all $g \in G$, we have ${\pi^*\Phi(g) \cong L^{-1} \otimes {\tau_g}^*(L)}$.
    \
  \end{sloppypar}
  \end{proposition}
\begin{proof}
  We have the Hoschild-Serre spectral sequence with ${E_{2}^{p,q}=\cohomology{q}{G}{\etalecohomology{p}{\xup}{\gmsheaf}}}$ that converges to $\etalecohomology{p+q}{\xdown}{\gmsheaf}$. Consequently, we have a long exact sequence as in Figure \ref{figure-long-exact-sequence}, part of which reads as follows:
$$ \cohomology{2}{G}{\gmsheaf(\xup)} \xrightarrow{\theta_4} \text{ker}\left( \cohomology{2}{\xdown}{\gmsheaf} \to  \cohomology{2}{\xup}{\gmsheaf}^{G}\right) \xrightarrow{\theta_5} \cohomology{1}{G}{\pic(\xup)} \xrightarrow{\theta_6} \cohomology{3}{G}{\gmsheaf(\xup)} $$
 By the above Proposition \ref{prop-group-hom-are-killed-by-theta-6}, we see that ${\{\pi^*\Phi(g)\}_{g \in G}}$ is a ${1\mbox{-cocycle}}$ representing a cohomology class in ${\cohomology{1}{G}{\pic(\xup)}}$ that lies in the kernel of the map $\theta_6$. By exactness, it corresponds to a cohomology class $\alpha$ in $\etalecohomology{2}{\xdown}{\gmsheaf}$. If $\cohomology{2}{G}{\gmsheaf(\xup)}$ is zero, then this class must be $n\mbox{-torsion}$ as $\cohomology{1}{G}{\pic(\xup)}$ is $n\mbox{-torsion}$. Thus, either of the conditions allows us to conclude that $\alpha$ is zero. Therefore the cocycle ${\{\pi^*\Phi(g)\}_{g \in G}}$ must be a coboundary as claimed.
\end{proof}
}
\section{Appendix I: Fibered categories and pullbacks} \label{section-fibered-cat-pullbacks}
{
\newcommand{\ztwo}{\overline{{\integers}/{2\integers}}}
\newcommand{\zconstant}{\underline{\integers}}
 
\begin{sloppypar}
  Suppose $\gerbe$ is a fibered category over a category $\basecat$. A choice of pullbacks (See \cite[\href{https://stacks.math.columbia.edu/tag/02XN}{Tag 02XN}]{stacks-project}) for the category $\gerbe$ is a choice of a strongly cartesian morphism $f^*x \to x$ lying over $f$, for every morphism $f: V \to U$ in $\basecat$, and an object $x \in \objects(\gerbe_{U})$ over $U$. After making a choice of pullbacks, we have a pullback functor ${f^*: \gerbe_{U} \to \gerbe_{V} }$. Suppose we have a morphism ${{F: \gerbe^{1} \to \gerbe^{2}}}$ of fibered categories over $\basecat$. Then for $U \in \objects(\basecat)$, we have the functor ${F_{U}: \gerbe^{1}_{U} \to \gerbe^{2}_{U}}$. We fix a choice of pullbacks for both $\gerbe^{1}$ and $\gerbe^{2}$. Then for a morphism ${f: V \to U}$ in $\basecat$, we have two functors:
  $$ F_V \circ f_1^* \text{ and } f_2^* \circ F_U : \gerbe^{1}_{U} \to \gerbe^{2}_{V}.$$
For ${y \in \objects(\gerbe^{1}_{U})}$, we have the following two strongly cartesian morphisms:
\end{sloppypar}
\[
  \begin{tikzcd}
    (f_2^* \circ F)y \arrow{r} \arrow[d,dashrightarrow] & Fy  \arrow[d,dashrightarrow] & (F \circ f_1^* )y \arrow{r} \arrow[d,dashrightarrow] & Fy \arrow[d,dashrightarrow] \\
    U \arrow{r}{f} & V & U \arrow{r}{f} & V 
  \end{tikzcd}  
\]
Since both diagrams are strongly cartesian, there exists a unique morphism
$${(f_2^* \circ F)y \to (F \circ f_1^* )y}$$ 
that connects the above two squares such that the resulting diagram commutes. In other words, we have a unique isomorphism of functors ${F_V \circ f_1^* \cong f_2^* \circ F_U}$. In this sense, we can pretend that the functor $F$ commutes with the pullback functors.

\section{Appendix II: Necessity of local  vanishing} \label{section-necessity-of-local-vanishing}
In this section, we construct examples to demonstrate that in the main Theorem \ref{theorem-induced-from-invariants-when-cohomology-class-is-zero}, the conditions on the local vanishing of group cohomology classes are necessary. In all the examples, we use the group $\autgroup$ which we denote by $G$.
\begin{remark}\label{remark-local-vanishing-imply-stalk-cohomology-free}
  \begin{sloppypar}
    In the situation of Definition \ref{def-group-cohomlogies-vanish-locally}, if $j^{th}$ group cohomology classes locally vanish, then for any stalk $A_p$ of $A$, the group cohomology ${\cohomology{j}{G}{A_p}}$ is zero.
  \end{sloppypar}
\end{remark}
\begin{example}\label{example-p1-square-map}
  \begin{sloppypar}
    Let ${\pi: \pline \to \pline}$ be the squaring map $x \to x^2$ from the projective line over $\complexns$ to itself. Let $\gmsheaf$  denote the sheaf of invertible functions on the small \etale{} site of $\pline$ and let $A$ denote the sheaf $\pi_*(\gmsheaf)$. Let $G$ act on ${\pline \xrightarrow{\pi} \pline}$ via the map $x \to -x$. This gives us an action of $G$ on $A$. Let $A^G$ denote the sheaf of invariants.
  \end{sloppypar}
\end{example}
\begin{sloppypar}
The sheaf $A^G$ is isomorphic to $\gmsheaf$ on the compliment of $\{0,\infty\}$ where $\pi$ is \etale{}. The stalk $A_0$ at zero is ${(\gmsheaf)_{0} \oplus (\gmsheaf)_{0}}$ on which $G$ acts trivially. Therefore we see that $\cohomology{1}{G}{A_0}$ is non-zero. Thus the first cohomology classes do not locally vanish near $0$ or $\infty$.
\end{sloppypar}
\begin{lemma}\label{lemma-p1-square-map-induced-map-on-h1}
  \begin{sloppypar}
    We have ${\etalecohomology{1}{\pline}{A} \cong \integers}$ and  ${\etalecohomology{1}{\pline}{A^G} \cong \integers}$. Moreover, the map ${\etalecohomology{1}{\pline}{A^G} \to \etalecohomology{1}{\pline}{A}}$ is given by ${\integers \xrightarrow{\times 2} \integers}$.
  \end{sloppypar}
\end{lemma}
\begin{proof}
Since $\pi$ is a finite morphism, the higher direct images along $\pi$ are zero (See \cite[\href{https://stacks.math.columbia.edu/tag/03QP}{Tag 03QP}]{stacks-project}). Therefore we have ${\cohomology{1}{\pline}{A} \cong \cohomology{1}{\pline}{\gmsheaf}}$ which is isomorphic to $\integers$. We have the following exact sequence of sheaves:
$$ 0 \to \gmsheaf \to \invariants \to \invariants / \gmsheaf \to 0 .$$
The sheaf $\invariants / \gmsheaf$ is a pushforward along the closed immersion $\{0,\infty\} \to \pline$. Since a closed immersion has no higher direct images, and since we are working over an algebraically closed field, we see that $A^G/\gmsheaf$ has no higher cohomology. Therefore from the long exact sequence of cohomologies arising from above exact sequence, we see that ${\cohomology{1}{\pline}{\invariants} \cong \cohomology{1}{\pline}{\gmsheaf} \cong \integers}$.
Thus, both the cohomology groups $\cohomology{1}{\pline}{\invariants}$ and $\cohomology{1}{\pline}{A}$ are isomorphic to $\pic(\pline)$. We omit some details, but the map between them corresponds to pulling back a line bundle along $\pi$ which is given by $l \to 2l$.
\end{proof}
The action of $G$ on $\cohomology{1}{\pline}{A}$ corresponds to pulling back a line bundle along the map ${x \to -x}$ and therefore it is trivial. The cohomology group $\cohomology{2}{G}{A(\pline)}$ is isomorphic to $\complexns^* / (\complexns^{*})^{2}$ which is zero. Let $y$ be an $A\mbox{-torsor}$ corresponding to a degree one line bundle. Then $y$ is stable under $G$ and therefore admits a lift of the $G$ action. Moreover, the obstruction class lies in the cohomology group $\cohomology{2}{G}{A(\pline)}$ which is zero. However, we see from the above lemma that $y$ is not induced from the invariants as the first group cohomology classes fail to vanish locally. 

\begin{example}\label{examples-circle-to-line}
  \begin{sloppypar}
    Let $S^1$ denote the unit circle on a cartesian plane. Let $X$ denote the diameter along the horizontal axis with end points $P$ and $Q$. We have the vertical projection map ${\pi: S^1 \to X}$. Let $\ztwo_{S^1}$ (respectively $\ztwo_{X}$) denote the constant sheaf with value $\autgroup$ on $S^1$ (respectively on $X$). Let $A$ denote the sheaf $\pi_*(\ztwo_{S^1})$ on $X$. We have an action of $G$ on $S^1$ given by the reflection along the horizontal axis which induces an action on the sheaf $A$.
  \end{sloppypar}
\end{example}
\begin{sloppypar}
  In the above example, for a sheaf $\mathscr{F}$ on $S^1$, the stalk $(\pi_*\mathscr{F})_x$ at a point $x \in X$ is ${\bigoplus_{s \in \pi^{-1}x} (\mathscr{F})_s}$. It follows that $\pi_*$ is exact and has no higher direct images. Therefore we have an isomorphism ${\cohomology{1}{X}{A} \cong \cohomology{1}{S^1}{\ztwo_{S^1}}}$ which is isomorphic to $\autgroup$. Consequently, we note that $G$ can only act trivially on $\cohomology{1}{X}{A}$. It is evident that we have $\invariants \cong \ztwo_{X}$ and therefore $\cohomology{1}{X}{\invariants}$ is zero. The group cohomology $\cohomology{1}{G}{A_{P}}$ for the stalk $A_{P}$ is $\autgroup$. Therefore the first group cohomology classes fail to vanish locally.
\end{sloppypar}
\begin{lemma}\label{lemma-obstruction-is-zero}
Let $y$ be the non-trivial $A\mbox{-torsor}$ on $S^1$. Then it admits a lift of the underlying action such that the obstruction class is zero. However, it is not induced from the invariants. 
\end{lemma}
\begin{proof}
Since $y$ represents a sheaf cohomology class in ${\cohomology{1}{\xdown}{A}}^G$, it admits a lift of the $G\mbox{-action}$. For a connected open set $U \subseteq X$ that contains either of the endpoints $P$ or $Q$, we have a $G\mbox{-equivariant}$ isomorphism $A(U) \cong A(X)$. Therefore we have an isomorphism ${\cohomology{2}{G}{A(U)} \cong \cohomology{2}{G}{A(X)}}$. Since the obstruction class locally vanishes everywhere, it is locally zero near $P$ or $Q$ and therefore is globally zero. There are no non-trivial $\invariants\mbox{-torsors}$ and therefore we conclude that $y$ is not induced from the invariants. 
\end{proof}
\begin{example}\label{examples-sphere-to-sphere}
  \begin{sloppypar}
    Let $S^2$ be the two dimensional sphere which we regard as the analytification of the complex projecive line. Let $\pi: S^2 \to S^2$ be the squaring map as in Example \ref{example-p1-square-map}. Let $\zconstant$ denote the constant sheaf  with value $\integers$ on $S^2$. Let $A$ denote the sheaf $\pi_*(\zconstant)$ on $S^2$. We have the action of $G$ on $S^2$ as in Example \ref{example-p1-square-map} and thus we get an action of $G$ on $A$. 
  \end{sloppypar}
\end{example}
\begin{sloppypar}
By the same reasoning as in Example \ref{examples-circle-to-line}, the higher direct images of $\pi_*$ are zero. Therefore we have ${\cohomology{2}{S^2}{A} \cong \cohomology{2}{S^2}{\zconstant}}$ which is isomorphic to $\integers$. It is easy to see that $A^G$ is isomorphic to the sheaf $\zconstant_{S^2}$. Thus both  $\cohomology{2}{S^2}{A}$ and $\cohomology{2}{S^2}{A^G}$ are isomorphic to the second cohomology group of $S^2$ which is $\integers$. Since $\pi$ is a map of degree two (See Chapter 2.2, \cite{hatcher2002algebraic}), the induced morphism ${\cohomology{2}{S^2}{\invariants} \to \cohomology{2}{S^2}{A}}$ is given by multiplication by two. The involution $z \to -z$ is homotopic to identity, and therefore the induced action of $G$ on $\cohomology{2}{S^2}{A}$ is trivial. Since $\pi$ is \etale{} away from $\{0,\infty\}$, the higher group cohomologies locally vanish on ${S^2 - \{0,\infty\}}$. For a connected open set $U$ containing $0$ or $\infty$, we have ${\cohomology{1}{G}{A(U)} \cong \integers[2]}$ which is zero. Thus, we see that the first group cohomology locally vanishes everywhere. However, the second cohomology $\cohomology{2}{G}{A_0}$ at the stalk at zero is $\autgroup$. Therefore the second group cohomology classes do not locally vanish near $0$ or $\infty$. Let $\gerbe$ be a gerbe banded by $A$ that is a generator of the group $\cohomology{2}{S^2}{A}$. Since $\gerbe$ is stable under the action of $G$, we have a lift of the $G\mbox{-action}$ on $\gerbe$. Since  $\cohomology{3}{G}{A(S^2)}$ is zero, we see that the obstruction class for the $G\mbox{-action}$ vanishes. However, $\gerbe$ is not induced from the invariants.
\end{sloppypar}

}
\bibliography{tropical_bib} 
\bibliographystyle{abbrv}
\end{document}